\documentclass[12pt]{elsarticlearxiv}

\usepackage{amsfonts}
\usepackage{amssymb,amsmath,amsthm,bm}

\usepackage{subfigure}
\usepackage{marvosym}
\RequirePackage[colorlinks,urlcolor=blue]{hyperref}
\usepackage{graphicx}
\usepackage{mathptmx} 
\usepackage{fancyhdr, lipsum}
\DeclareMathAlphabet{\mathcal}{OMS}{cmsy}{m}{n}
\numberwithin{equation}{section}
\newtheorem{thm}{Theorem}[section]
\newtheorem{cor}[thm]{Corollary}
\newtheorem{lem}[thm]{Lemma}
\newtheorem{prop}[thm]{Proposition}

\newtheorem{defn}[thm]{Definition}
\newtheorem{ass}[thm]{Assumption}
\newtheorem{ex}[thm]{Example}
\newtheorem{rem}[thm]{Remark}

\newcommand{\tr}{\rm tr}

\newcommand{\G}{\mathcal{N}}
\newcommand{\R}{\mathbb{R}}
\newcommand{\rp}{\mathbb{P}}
\newcommand{\ra}{\mathcal{R}}
\newcommand{\D}{\mathcal{D}}
\newcommand{\M}{\mathcal{M}_q}
\newcommand{\Mr}{\mathcal{M}_r}
\newcommand{\Om}{\Omega}

\newcommand{\hiddensection}[1]{
\stepcounter{section}
\section*{}}

\newcommand{\Sr}{(\A_0^\frac14+\mathcal{M}_r)}
\newcommand{\Sb}{(\A_0^\frac{\beta\alpha}2+\mathcal{M}_r)}

\newcommand{\Sq}{(\A_0+\mathcal{M}_q)}

\newcommand{\N}{\mathbb{N}}

\newcommand{\A}{\mathcal{A}}
\newcommand{\C}{\mathcal{C}}

\newcommand{\Ss}{\mathcal{S}}
\newcommand{\B}{\mathcal{B}}

\newcommand{\X}{\mathcal{X}}

\newcommand{\T}{\mathcal{T}}
\newcommand{\Sc}{\mathcal{S}}

\newcommand{\coloneq}{\mathrel{\mathop:}=}
\newcommand{\eqcolon}{=\mathrel{\mathop:}}
\newcommand{\norm}[1]{\big\|#1\big\|}
\newcommand{\sumi}{\sum_{k=1}^{\infty}}
\newcommand{\E}{\mathbb{E}}
\newcommand{\pr}[2]{\big\langle#1,#2\big\rangle}
\newcommand{\envelope}{(\raisebox{-.5pt}{\scalebox{1.45}{\Letter}}\kern-1.7pt)}

\begin{document}

\begin{frontmatter}



\title{\small{Preprint of final article which will appear as
Stochastic Processes and their Applications
Volume 123, Issue 10, October 2013, Pages 3828-3860}\\\vspace{0.5cm}\Large{Posterior Contraction Rates for the Bayesian Approach to Linear Ill-Posed Inverse Problems}
}

\author[m1]{Sergios Agapiou\corref{cor1}}
\ead{S.Agapiou@warwick.ac.uk}
\cortext[cor1]{Corresponding author.}
\author[m2]{Stig Larsson}
\ead{Stig@chalmers.se} 
\author[m1]{Andrew M. Stuart}
\ead{A.M.Stuart@warwick.ac.uk}
%\affiliation{???}

%\runauthor{S. Agapiou, S. Larsson and A.M. Stuart}
\address[m1]{Mathematics Institute, University of Warwick\\  Coventry CV4 7AL, United Kingdom}

\address[m2]{Department of Mathematical Sciences, Chalmers University of Technology\\ and University of Gothenburg, SE-412 96 Gothenburg, Sweden}

\begin{abstract}
We consider a Bayesian nonparametric approach to a family of linear inverse problems in a separable Hilbert space setting with Gaussian noise. We assume Gaussian priors, which are conjugate to the model, and present a method of identifying the posterior using its precision operator. Working with the unbounded precision operator enables us to use partial differential equations (PDE) methodology to obtain rates of contraction of the posterior distribution to a Dirac measure centered on the true solution. Our methods assume a relatively weak relation between the prior covariance, noise covariance and forward operator, allowing for a wide range of applications.\end{abstract}

\begin{keyword}posterior consistency \sep posterior contraction \sep Gaussian prior \sep posterior distribution \sep inverse problems
\MSC 62G20 \sep 62C10 \sep 35R30 \sep 45Q05



\end{keyword}

\end{frontmatter}

% \linenumbers

\section{Introduction}\label{sec:intro}
The solution of inverse problems provides a rich source of 
applications of the Bayesian nonparametric methodology. It encompasses a broad range
of applications from partial differential equations (PDEs)~\cite{MR1045629}, where there is a well-developed theory of classical, non-statistical, regularization~\cite{MR1408680}. On the other hand, the area of nonparametric Bayesian statistical estimation and in particular the problem of posterior consistency has attracted a lot of interest in recent years; see for instance~\cite{MR1790007, MR1865337, MR2589319,MR2418663,MR2357712,gine,MR829555}. Despite this, the formulation of many of these PDE inverse problems using the Bayesian approach is in its infancy~\cite{MR2652785}. 
Furthermore, the development of a theory of Bayesian posterior consistency, analogous to the theory for classical regularization, is under-developed with the primary contribution being the recent paper~\cite{1232.62079}.
This recent paper provides a roadmap for what is to be expected regarding Bayesian posterior consistency, but is limited in terms of applicability by the
assumption of simultaneous diagonalizability of the three linear operators required to define Bayesian inversion. Our aim in this paper is to make a significant step in the theory of Bayesian posterior consistency for linear inverse problems by developing a methodology which sidesteps the need for simultaneous diagonalizability. The central idea underlying the analysis is to work
with precision operators rather than covariance operators, and thereby to enable use of powerful tools from PDE theory to facilitate the analysis.

      Let $\X$ be a separable Hilbert space, with norm $\|\cdot\|$ and inner product $\langle\cdot,\cdot\rangle$, and let $\A\colon \D(\A)\subset\X\to\X$ be a known  self-adjoint and positive-definite linear operator with bounded inverse.       We consider the inverse problem to find $u$ from  $y$, where $y$ is a noisy observation of $\A^{-1}u$. We assume the model, \begin{equation}\label{eq:int1}y=\A^{-1}u+\frac{1}{\sqrt{n}}\xi, \end{equation}      
      where $\frac{1}{\sqrt{n}}\xi$ is an additive noise. We will be particularly interested in the small noise limit where $n\to\infty$.

A popular method in the deterministic approach to inverse problems is the generalized Tikhonov-Phillips regularization method in which  $u$ is approximated by the minimizer of a regularized least squares functional: define the Tikhonov-Phillips 
 functional \begin{equation}\label{eq:int2}J_0(u)\coloneq \frac12\norm{\C_1^{-\frac12}(y-\A^{-1}u)}^2+\frac{\lambda}{2}\norm{\C_0^{-\frac12}u}^2,\end{equation} where $\C_i\colon \X\to\X, \;i=0,1,$ are bounded, possibly compact, self-adjoint positive-definite linear operators. The parameter $\lambda$ is called the regularization parameter, 
 and in the classical non-probabilistic approach the general practice is to choose it as an appropriate function of the noise size $n^{-\frac12}$, which shrinks to zero as $n\to\infty$, in order to recover the unknown parameter $u$~\cite{MR1408680}.

In this paper we adopt a Bayesian approach for the solution of problem (\ref{eq:int1}), which will be linked to the minimization of $J_0$ via the posterior mean. We assume that the prior distribution is Gaussian, $u\sim\mu_0=\G(0,\tau^2\C_0)$, where $\tau>0$ and $\C_0$ is a self-adjoint, positive-definite, trace class, linear operator on $\X$. We also assume that the noise is  Gaussian, $\xi\sim\G(0,\C_1)$, where $\C_1$ is a self-adjoint positive-definite, bounded, but not necessarily trace class, linear operator; this allows us to include the case of white observational noise. We assume that the, generally unbounded, operators $\C_0^{-1}$ and $\C_1^{-1},$ have been maximally extended to self-adjoint positive-definite operators on appropriate domains. The unknown parameter and the noise are considered to be independent, thus the conditional distribution of the observation given the unknown parameter $u$ (termed the likelihood) is also Gaussian with distribution $y|u\sim\G(\A^{-1}u,\frac1n\C_1).$ 
        
Define $\lambda=\frac1{n\tau^2}$ and let \begin{equation}\label{eq:int3}J(u)=nJ_0(u)=\frac{n}2\norm{\C_1^{-\frac12}(y-\A^{-1}u)}^2+\frac1{2\tau^2}\norm{\C_0^{-\frac12}u}^2.\end{equation}
In finite dimensions the probability density of the posterior distribution, that is, the distribution of the unknown given the observation, with respect to the Lebesgue measure is proportional to $\exp{\left(-J(u)\right)}.$ This suggests that, in the infinite-dimensional setting, the posterior is Gaussian $\mu^y=\G(m,\C)$, where we can identify the posterior covariance and mean by the equations \begin{equation}\label{eq:int4}\C^{-1}=n\A^{-1}\C_1^{-1}\A^{-1}+\frac1{\tau^2}\C_0^{-1} \end{equation}
and \begin{equation}\label{eq:int5}\frac1n\C^{-1}m=\A^{-1}\C_1^{-1}y, \end{equation}obtained by completing the square. We present a method of justifying these expressions in Section \ref{sec:justification}. We define \begin{equation}\label{eq:int6}\B_\lambda=\frac1n\C^{-1}=\A^{-1}\C_1^{-1}\A^{-1}+\lambda \C_0^{-1}\end{equation}
and observe that the dependence of $\B_\lambda$ on $n$ and $\tau$ is only through $\lambda$. Since \begin{equation}\label{eq:int7}\B_\lambda m=\A^{-1}\C_1^{-1}y, \end{equation}  the posterior mean also depends only on $\lambda$: $m=m_\lambda$. This is not the case for the posterior covariance $\C$, since it depends on $n$ and $\tau$ separately:  $\C=\C_{\lambda,n}$. In the following, we suppress the dependence of the posterior covariance on $\lambda$ and $n$ and we denote it by $\C$. 

Observe that the posterior mean is the minimizer of the functional $J$, hence also of $J_0,$ that is, the posterior mean is the Tikhonov-Phillips regularized approximate solution of problem (\ref{eq:int1}), for the functional $J_0$ with $\lambda=\frac{1}{n\tau^2}$.

In~\cite{MR731228} and~\cite{MR1009041}, formulae for the posterior covariance and mean are identified in the infinite-dimensional setting, which avoid using any of the inverses of the prior, posterior or noise covariance operators.  They obtain  \begin{equation}\label{eq:int8}\C=\tau^2\C_0-\tau^2\C_0\A^{-1}(\A^{-1}\C_0\A^{-1}+\lambda \C_1)^{-1}\A^{-1}\C_0\end{equation}and \begin{equation}\label{eq:int9}m=\C_0\A^{-1}(\A^{-1}\C_0\A^{-1}+\lambda \C_1)^{-1}y,\end{equation} which are consistent with formulae (\ref{eq:int4}) and (\ref{eq:int7}) for the finite-dimensional case. In~\cite{MR731228} this is done only for $\C_1$ of trace class while in~\cite{MR1009041} the case of white observational noise was included. We will work in an infinite-dimensional setting where the formulae (\ref{eq:int4}), (\ref{eq:int7}) for the posterior covariance and mean can be justified. Working with the unbounded operator $\B_\lambda$ opens the possibility of using tools of analysis, and also numerical analysis, familiar from the theory of partial differential equations.

In our analysis we always assume that $\C_0^{-1}$ is regularizing, that is, we assume that $\C_0^{-1}$ dominates $\B_\lambda$ in the sense that it induces stronger norms than $\A^{-1}\C_1^{-1}\A^{-1}.$ This is a reasonable assumption since otherwise we would have $\B_\lambda\simeq \A^{-1}\C_1^{-1}\A^{-1}$ (here $\simeq$ is used loosely to indicate two operators which induce equivalent norms; we will make this notion precise in due course).  This would imply that the posterior mean is $m\simeq\A y$, meaning that we attempt to invert the data by applying the, generally discontinuous, operator $\A$~\cite[Proposition 2.7]{MR1408680}.

We study the consistency of the posterior $\mu^y$ in the frequentist setting. To this end, we consider data $y=y^\dagger$ which is a realization of  \begin{equation}\label{eq:int10}y^{\dagger}=\A^{-1}u^\dagger+\frac1{\sqrt{n}}\xi, \quad\xi\sim\G(0,\C_1),\end{equation}
where $u^\dagger$ is a fixed element of $\X$; that is, we consider observations which are perturbations of the image of a fixed true solution $u^\dagger$ by an additive noise $\xi$, scaled by $\frac1{\sqrt{n}}$. Since the posterior depends through its mean on the data and also through its covariance operator on the scaling of the noise and the prior, this choice of data model gives as posterior distribution the Gaussian measure $\mu^{y^\dagger}_{\lambda, n}=\G(m_\lambda^\dagger, \C)$, where $\C$ is given by (\ref{eq:int4}) and \begin{equation}\label{eq:int11}\B_\lambda m_\lambda^\dagger=\A^{-1}\C_1^{-1}y^\dagger. \end{equation}
We study the behavior of the posterior $\mu^{y^\dagger}_{\lambda,n}$ as the noise disappears ($n\to\infty$). Our aim is to show that it contracts to a Dirac measure centered on the fixed true solution $u^\dagger$. In particular, we aim to determine $\varepsilon_n$ such that \begin{equation}\label{eq:main1}\E^{y^\dagger}\mu^{y^\dagger}_{\lambda,n}\left\{u:\norm{u-u^\dagger}\geq M_n\varepsilon_n\right\}\to0, \quad \forall M_n\to\infty,\end{equation}
where the expectation is with respect to the random variable $y^\dagger$ distributed according to the data likelihood
 $\G(\A^{-1}u^\dagger, \frac1n\C_1)$.
 
As in the deterministic theory of inverse problems, in order to get convergence in the small noise limit, we let the regularization disappear in a carefully chosen way, that is, we will choose $\lambda=\lambda(n)$ such that $\lambda\to0$ as $n\to\infty$. The assumption that $\C_0^{-1}$ dominates $\B_\lambda$,
shows that $\B_\lambda$  is a singularly perturbed unbounded (usually differential) operator, with an inverse which blows-up in the limit $\lambda\to0$. This together with equation (\ref{eq:int7}), opens up the possibility of using the analysis of such singular limits to study posterior contraction: on the one hand, as $\lambda\to0$, $\B_\lambda^{-1}$ becomes unbounded; on the other hand, as $n\to\infty$, we have more accurate data, suggesting that for the appropriate choice of $\lambda=\lambda(n)$ we can get $m_\lambda^\dagger\simeq u^\dagger$. In particular, we will choose $\tau$ as a function of the scaling of the noise, $\tau=\tau(n),$  under the restriction that the induced choice of $\lambda=\lambda(n)=\frac{1}{n\tau(n)^2}$, is such that $\lambda\to0$ as $n\to\infty$. The last choice will be made in a way which optimizes the rate of posterior contraction $\varepsilon_n$, defined in (\ref{eq:main1}). In general there are three possible asymptotic behaviors of the scaling of the prior $\tau^2$ as $n\to\infty$,~\cite{MR2357712,1232.62079}:
\begin{enumerate}
\item[i)]$\tau^2\to\infty$; we increase the prior spread, if we know that  draws from the prior are more regular than $u^\dagger$;
\item[ii)]$\tau^2$ fixed; draws from the prior have the same regularity as $u^\dagger$;
\item[iii)]$\tau^2\to0$ at a rate slower than $\frac1n$; we shrink the prior spread, when we know that draws from the prior are less regular than $u^\dagger.$
\end{enumerate} 

The problem of posterior contraction in this context is also investigated in~\cite{1232.62079} and~\cite{simoni2}. In~\cite{1232.62079}, sharp convergence rates are obtained in the case where $\C_0, \C_1$ and $\A^{-1}$ are simultaneously diagonalizable, with eigenvalues decaying algebraically, and in particular $\C_1=I$, that is, the data are polluted by white noise. In this paper we relax the assumptions on the relations between the operators $\C_0, \C_1$ and $\A^{-1}$, by assuming that appropriate powers of them induce comparable norms (see Section \ref{sec:assumptions}). In~\cite{simoni2}, the non-diagonal case is also examined; the three operators involved are related through domain inclusion assumptions. The assumptions made in~\cite{simoni2} can be quite restrictive in practice; our assumptions include settings not covered in~\cite{simoni2}, and in particular the case of white observational noise.

\subsection{Outline of the rest of the paper}
In the following section we present our main results which concern the identification of the posterior (Theorem \ref{justth}) and the posterior contraction (Theorems \ref{pdecor1} and \ref{pdecor2}). In Section \ref{sec:assumptions} we present our assumptions and their implications. The proofs of the main results are built in a series of intermediate results contained in Sections \ref{sec:prmean}-\ref{sec:main}. In Section \ref{sec:prmean}, we reformulate equation (\ref{eq:int7}) as a weak equation in an infinite-dimensional space. In Section \ref{sec:justification}, we present a new method of identifying the posterior distribution: we first characterize it through its Radon-Nikodym derivative with respect to the prior (Theorem \ref{prop0}) and then justify the formulae (\ref{eq:int4}), (\ref{eq:int7}) for the posterior covariance and mean (proof of Theorem \ref{justth}). In Section \ref{sec:normbounds}, we present operator norm bounds for $\B_\lambda^{-1}$ in terms of the singular parameter $\lambda$, which are the key to the posterior contraction results contained in Section \ref{sec:main}  and their corollaries in Section \ref{sec:mainresults} (Theorems \ref{pdeth1}, \ref{pdeth2} and \ref{pdecor1}, \ref{pdecor2}). In Section \ref{sec:ex}, we present some nontrivial examples satisfying our assumptions and provide the corresponding rates of convergence. In Section \ref{sec:diag}, we compare our results to known minimax rates of convergence in the case where $\C_0, \C_1$ and $\A^{-1}$ are all diagonalizable in the same eigenbasis and have eigenvalues that decay algebraically. Finally, Section \ref{sec:conclusions} is a short conclusion.

The entire paper rests on a rich
set of connections between the theory of stochastic processes
and various aspects of the theory of linear partial differential
equations. In particular, since the Green's function of the
precision operator of a Gaussian measure corresponds to its
covariance function, our formulation and analysis of the
inverse problem via precision operators is very natural.
Furthermore, estimates on the inverse of  singular limits
of these precisions, which have direct implications for
localization of the Green's functions, play a key role in the
analysis of posterior consistency.

%%%%%%%%%%%%%%% Section Main Results %%%%%%%%%%%%%%%%

\section{Main Results}\label{sec:mainresults}
In this section we present our main results. We postpone the rigorous presentation of our assumptions to the next section and the proofs and technical lemmas are presented together with intermediate results of independent interest in Sections \ref{sec:prmean} - \ref{sec:main}. Recall that we assume a Gaussian prior $\mu_0=\G(0,\tau^2\C_0)$ and a Gaussian noise distribution $\G(0,\C_1)$. Our first assumption concerns the decay of the eigenvalues of the prior covariance operator and enables us to quantify the regularity of draws from the prior. This is encoded in the parameter $s_0\in[0,1)$; smaller $s_0$ implies more regular draws from the prior. We also assume that $\C_1\simeq\C_0^\beta$ and $\A^{-1}\simeq\C_0^\ell$, for some $\beta,\ell\geq0$, where $\simeq$ is s used in the manner outlined in
Section \ref{sec:intro}, and defined in detail in Section \ref{sec:assumptions}. Finally, we assume that the problem is sufficiently ill-posed with respect to the prior. This is quantified by the parameter $\Delta\coloneq 2\ell-\beta+1$ which we assume to be larger than $2s_0$; for a fixed prior, the larger $\Delta$ is, the more ill-posed the problem.
\subsection{Posterior Identification}
Our first main theorem identifies the posterior measure as Gaussian and justifies formulae \ref{eq:int4} and \ref{eq:int7}. This reformulation of the posterior in terms of the
precision operator is key to our method of analysis
of posterior consistency and opens the route to using methods from
the study of partial differential equations (PDEs). These methods
will also be useful for the development of numerical methods for
the inverse problem.\begin{thm}\label{justth}
Under the Assumptions \ref{a2}, the posterior measure $\mu^y(du)$ is  Gaussian $\mu^y=\G(m,\C)$, where $\C$ is given by (\ref{eq:int4}) and $m$ is a weak solution of (\ref{eq:int7}). \qed
\end{thm}

\subsection{Posterior Contraction}
We now present our results concerning frequentist posterior consistency of the Bayesian solution to the inverse problem. We assume to have data $y^\dagger=y^\dagger(n)$ as in (\ref{eq:int10}), and examine the behavior of the posterior $\mu^{y^\dagger}_{\lambda,n}=\G(m_\lambda^\dagger, \C)$, where $m_\lambda^\dagger$ is given by (\ref{eq:int11}), as the noise disappears ($n\to\infty$). The first convergence result concerns the convergence of the posterior mean $m_\lambda^\dagger$ to the true solution $u^\dagger$ in a range of weighted norms $\|\cdot\|_\eta$ induced by powers of the prior covariance operator $\C_0$. The spaces $(X^\eta,\|\cdot\|_\eta)$ are rigorously defined in the following section. The second result provides rates of posterior contraction of the posterior measure to a Dirac centered on the true solution as described in (\ref{eq:main1}). In both results, we assume a priori known regularity of the true solution $u^\dagger\in X^\gamma$ and give the convergence rates as functions of $\gamma$.

\begin{thm}\label{pdecor1}
Assume $u^\dagger\in X^\gamma$, where $\gamma\geq1$ and let $\eta=(1-\theta)(\beta-2\ell)+\theta$, where $\theta\in[0,1]$. Under the Assumptions \ref{a2}, we have the following optimized rates of convergence, where $\varepsilon>0$ is arbitrarily small:
\begin{enumerate}
\item[i)]if $\gamma\in(1,\Delta+1],$ for $\tau=\tau(n)=n^{-\frac{\gamma-1+s_0+\varepsilon}{2(\Delta+\gamma-1+s_0+\varepsilon)}}$ \begin{equation}\E^{y^\dagger}\norm{m_\lambda^\dagger-u^\dagger}_{\eta}^2\leq cn^{-\frac{\Delta+\gamma-1-\theta\Delta}{\Delta+\gamma-1+s_0+\varepsilon}}; \nonumber\end{equation}
\item[ii)]if $\gamma>\Delta+1$, for $\tau=\tau(n)=n^{-\frac{\Delta+s_0+\varepsilon}{2(2\Delta+s_0+\varepsilon)}}$ \begin{equation}\E^{y^\dagger}\norm{m_\lambda^\dagger-u^\dagger}_{\eta}^2\leq cn^{-\frac{(2-\theta)\Delta}{2\Delta+s_0+\varepsilon}}; \nonumber\end{equation}
\item[iii)]if $\gamma=1$ and $\theta\in[0,1)$ for $\tau=\tau(n)=n^{-\frac{s_0+\varepsilon}{2(\Delta+s_0+\varepsilon)}}$ \begin{equation}\E^{y^\dagger}\norm{m_\lambda^\dagger-u^\dagger}_{\eta}^2\leq cn^{-\frac{(1-\theta)\Delta}{\Delta+s_0+\varepsilon}}. \nonumber\end{equation}If $\gamma=1$ and $\theta=1$ then the method does not give convergence.
\end{enumerate}\qed
\end{thm}

\begin{thm}\label{pdecor2}Assume $u^\dagger\in X^\gamma$, where $\gamma\geq1$. Under the Assumptions \ref{a2}, we have the following optimized rates for the convergence in (\ref{eq:main1}), where $\varepsilon>0$ is arbitrarily small:
\begin{enumerate}
\item[i)]if $\gamma\in[1,\Delta+1]$ for $\tau=\tau(n)=n^{-\frac{\gamma-1+s_0+\varepsilon}{2(\Delta+\gamma-1+s_0+\varepsilon)}}$
\begin{equation}\varepsilon_n=\left\{\begin{array}{ll}n^{-\frac{\gamma}{2(\Delta+\gamma-1+s_0+\varepsilon)}} , & if 
\;\mbox{$\beta-2\ell\leq0$} 
                                     \\ n^{-\frac{\Delta+\gamma-1}{2(\Delta+\gamma-1+s_0+\varepsilon)}}, & otherwise;
                                    \end{array}\right.\nonumber\end{equation}

\item[ii)]if $\gamma>\Delta+1$  for $\tau=\tau(n)=n^{-\frac{\Delta+s_0+\varepsilon}{2(2\Delta+s_0+\varepsilon)}}$

\begin{equation}\varepsilon_n=\left\{\begin{array}{ll}n^{-\frac{\Delta+1}{2(2\Delta+s_0+\varepsilon)}} , & if 
\;\mbox{$\beta-2\ell\leq0$} 
                                     \\ n^{-\frac{\Delta}{2\Delta+s_0+\varepsilon}}, & otherwise.
                                    \end{array}\right.\nonumber\end{equation}
\end{enumerate}\qed
\end{thm}

To summarize, provided the problem is sufficiently ill-posed and the true solution $u^\dagger$ is sufficiently regular we get the convergence in (\ref{eq:main1}) for \begin{equation}\varepsilon_n=n^{-\frac{\gamma\wedge(\Delta+1)}{2(\Delta+\gamma\wedge(\Delta+1)-1+s_0+\varepsilon)}}.\nonumber\end{equation} %The parameters $\gamma$, $s_0$ and $\Delta$ measure the regularity of the true solution, the regularity of the prior and the ill-posedness of the problem respectively, as explained in Sections \ref{sec:assumptions} and \ref{sec:main}.

\begin{figure}[htp]
            \begin{center}\includegraphics[type=pdf, ext=.pdf, read=.pdf, width=90mm]{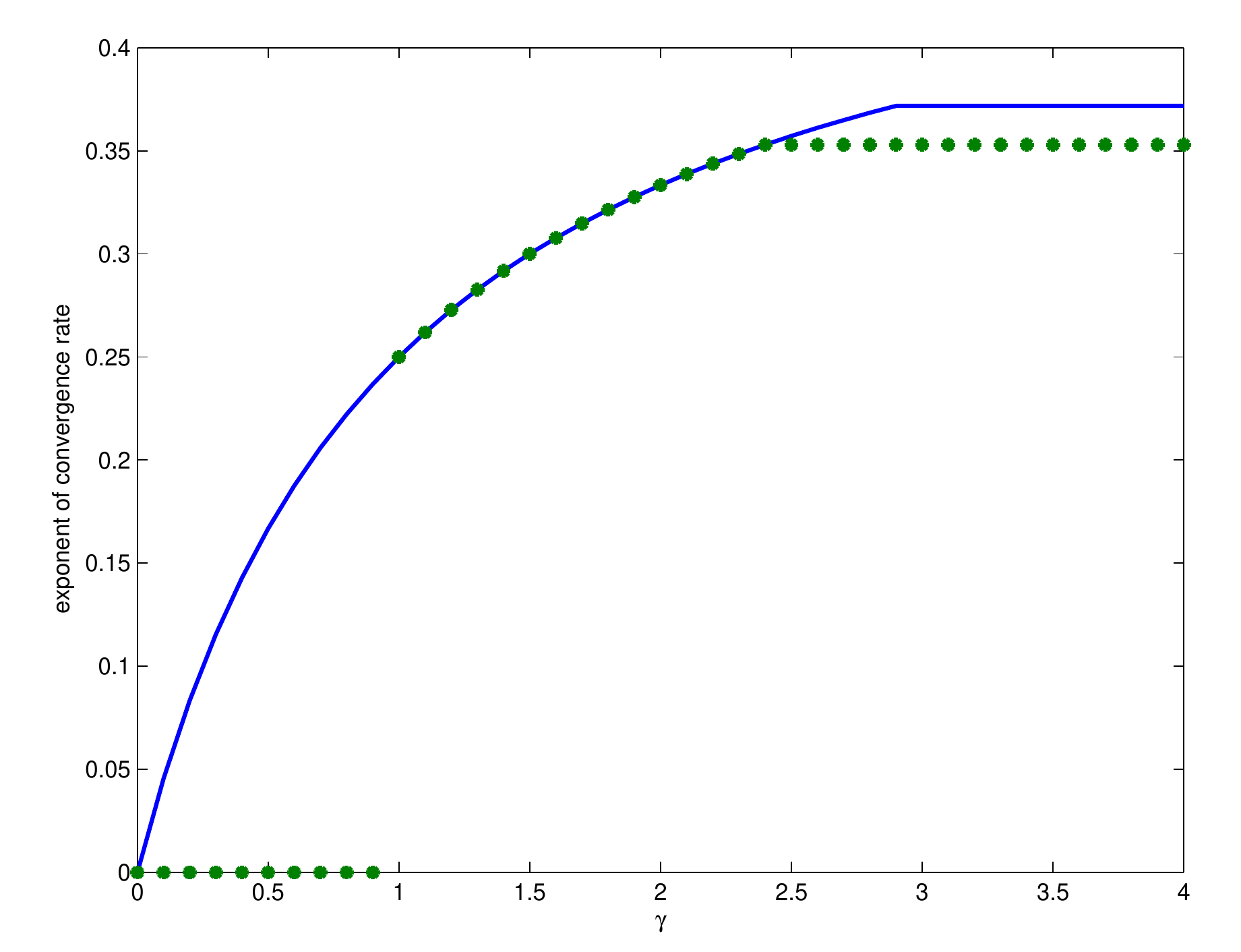} \caption{Exponents of rates of contraction plotted against the regularity of the true solution, $\gamma$. In blue are the sharp convergence rates obtained in the diagonal case in~\cite{1232.62079}, while in green the rates predicted by our method, which applies to the more general non-diagonal case}
          \label{fig}\end{center}
\end{figure}
Our rates of convergence agree, up to $\varepsilon>0$ arbitrarily small, with the sharp convergence rates obtained in the diagonal case in~\cite{1232.62079} across a wide range of regularity assumptions on the true solution (Figure \ref{fig}); yet, our rates cover a much more applicable range of non-simultaneously diagonalizable problems. (The reason for the appearance of $\varepsilon$ is that in the assumed non-diagonal setting we can only use information about the regularity of the noise as expressed in terms of the spaces $X^\rho$ (cf. Lemma \ref{l2}), rather than the explicit representation of the noise.)

The rates we obtain are not as strong as in the simultaneously diagonalizable case when the true solution is too regular; in particular our rates saturate earlier as a function of increasing regularity, and we require a certain degree of regularity of the true solution in order to secure convergence. 
It is not known if our results can be improved but it would be interesting to try. Both of the two discrepancies 
are attributed to the fact that our method relies on interpolating between rates in a strong and a weak norm of the error $e=m_\lambda^\dagger-u^\dagger$; on the one hand the rate of the error in the weak norm saturates earlier, and on the other hand the error in the strong norm requires additional regularity in order to converge (cf. Section \ref{sec:diag}).

%Even though we work in a considerably more difficult non-diagonal setting, we get rates which, up to $\varepsilon>0$ arbitrarily small, agree with the sharp convergence rates obtained in the diagonal case in~\cite{1232.62079} for a wide range of values of the parameter $\gamma$ (Figure \ref{fig}). The reason for the appearance of $\varepsilon$ in our rates is the fact that in the non-diagonal case we cannot use the explicit representation of the noise as in~\cite{1232.62079}. We are only able to use the support of the noise distribution, that is, the spaces $X^\rho$ in which the noise lives with probability 1. Unfortunately, the parameter $\rho$ belongs to an open interval of the form $(-\infty,\beta-s_0)$ (see Lemma \ref{l2} below) which only allows us to exploit the regularity of the noise up to an $\varepsilon>0$ arbitrarily small. For this reason, we do not know if it is possible the rates to be sharp at this level of generality. Our rates agree with the diagonal case acrossa a wide range of regularity assumptions, yet cover a much more Our rates fail to be optimal when the true solution is too regular, in particular our rates saturate earlier, and we also require more regularity from the true solution in order to get convergence. Both of these discrepancies  Despite these disadvantages, we hope that our method will be a starting point for expanding to optimally cover the other cases. Our method is also applicable to other problems, for example nonparametric drift estimation as presented in~\cite{Pokern:2012fk}. 

%%%%%%%%%%%%%%%% Section Assumptions %%%%%%%%%%%%%%%

\section{The Setting}\label{sec:assumptions}
In this section we present the setting in which we formulate our results.
First, we define the spaces in which we work, in particular, we define the Hilbert scale induced by the prior covariance operator $\C_0$. Then we define the probability measures relevant to our analysis. Furthermore, we state our main assumptions, which concern the decay of the eigenvalues of $\C_0$ and the connections between the operators $\C_0$, $\C_1$  and $ \A^{-1}$, and present
 regularity results for draws from the prior, $\mu_0$, and the noise distribution, $\G(0,\C_1)$.
Finally we briefly overview the way in which the Hilbert scale
defined in terms of the prior covariance operator $\C_0$, which
is natural for our analysis, links to scales of spaces
defined independently of any prior model.

We start by defining the Hilbert scale which we will use in our analysis. Recall that $\X$ is an infinite-dimensional
separable Hilbert space and $\C_0\colon\X\to\X$ is a self-adjoint, positive-definite, trace class, linear operator. Since $\C_0\colon \X\to\X$ is injective and self-adjoint we have that $\X=\overline{\ra(\C_0)}\oplus\ra(\C_0)^{\perp}=\overline{\ra(\C_0)}$. This means that $\C_0^{-1}\colon\ra(\C_0)\to\X$ is a densely defined, unbounded, symmetric, positive-definite, linear operator in $\X$. Hence it can be extended to a self-adjoint operator with domain $\D(\C_0^{-1})\coloneq\{u\in\X: \C_0^{-1}u\in\X\}$; this is the Friedrichs extension~\cite{MR1892228}.
Thus, we can define the Hilbert scale $(X^t)_{t\in\R}$, with $X^t\coloneq \overline{\mathcal{M}}^{{\|.\|}_{t}}$~\cite{MR1408680}, where\begin{equation}\mathcal{M}\coloneq \bigcap_{l=0}^\infty\D(\C_0^{-l}),\;\pr{u}{v}_{t}\coloneq \big\langle\C_0^{-\frac{t}2}u,\C_0^{-\frac{t}2}v\big\rangle \quad\text{and}\quad {\|u\|}_{t}\coloneq \norm{\C_0^{-\frac{t}2}u}.\nonumber\end{equation}
The bounded linear operator $\C_1\colon\X\to\X$ is assumed to be self-adjoint, positive-definite (but not necessarily trace class); thus $\C_1^{-1}\colon\ra(\C_1)\to\X$ can be extended in the same way to a self-adjoint operator with domain $\D(\C_1^{-1})\coloneq\{u\in\X: \C_1^{-1}u\in\X\}.$ Finally, recall that we assume that $\A\colon\D(\A)\to\X$ is a self-adjoint and positive-definite, linear operator with bounded inverse, $\A^{-1}\colon\X\to\X$.

We assume that we have a probability space $(\Omega,\mathcal{F},\mathcal{P})$. 
The expected value is denoted by $\E$ and $\xi\sim\mu$ means that the law of the random variable $\xi$ is the measure $\mu$.

Let $\mu_0:=\G(0,\tau^2\C_0)$ and $\rp_0:=\G(0,\frac1n\C_1)$ be the prior and noise distributions respectively. Furthermore, let $\nu(du,dy)$ denote the measure constructed by taking $u$ and $y|u$ as independent Gaussian random variables $\G(0,\tau^2\C_0)$ and $\G(\A^{-1}u,\frac1n\C_1)$ respectively:
 \begin{equation}\nu(du,dy)=\rp(dy|u)\mu_0(du),\nonumber\end{equation} where $\rp\coloneq\G(\A^{-1}u,\frac1n\C_1)$. We denote by $\nu_0(du,dy)$  the measure constructed by taking $u$ and $y$ as independent Gaussian random variables $\G(0,\tau^2\C_0)$ and $\G(0,\frac1n\C_1)$ respectively: \begin{equation}\nu_0(du,dy)=\rp_0(dy)\otimes\mu_0(du).\nonumber\end{equation}

Let $\{\lambda_k^2, \phi_k\}_{k=1}^\infty$ be orthonormal eigenpairs of $\C_0$ in $\X$. Thus, $\{\lambda_k\}_{k=1}^\infty$ are the singular values and $\{\phi_k\}_{k=1}^\infty$ an orthonormal eigenbasis. Since $\C_0$ is trace class we have that $\sumi\lambda_k^2<\infty$. In fact we require a slightly stronger assumption see Assumption \ref{a2}(\ref{a1}) below.

\subsection{Assumptions}
We are now ready to present our assumptions. The first assumption enables us to quantify the regularity of draws from the prior whereas the rest of the assumptions regard interrelations between the three operators $\C_0$, $\C_1$ and $\A^{-1}$; these assumptions reflect the idea that  \begin{equation}\C_1\simeq \C_0^{\beta}\quad \text{and}\quad \A^{-1}\simeq \C_0^\ell,\nonumber\end{equation} for some $\beta\geq0, \ell\geq0$, where $\simeq$ is used in the same manner as in Section \ref{sec:intro}. This is made precise by the inequalities presented in the following assumption, where the notation  $a \asymp b$ means that there exist constants $c, c'>0$ such that $ca\le b\le c' a$. 

\begin{ass}\label{a2}Suppose there exist $s_0\in[0,1)$, $\beta\geq0, \;\ell\geq0$ and constants $c_i>0, i=1,..,4$ such that\begin{enumerate}
\item\label{a1}$\C_0^s$ is trace class for all $s>s_0$;
\item\label{a2delta} $\Delta>2 s_0$, where $\Delta\coloneq2\ell-\beta+1;$
\item\label{a2i}$\norm{\C_1^{-\frac12}\A^{-1}u}\asymp \norm{\C_0^{\ell-\frac{\beta}2}u}, \quad\forall u\in X^{\beta-2\ell};$
\item\label{a2ii} $\norm{\C_0^{-\frac{\rho}2}\C_1^{\frac12}u}\leq c_1\norm{\C_0^\frac{\beta-\rho}2u}, \quad\forall u\in X^{\rho-\beta}, \;\forall \rho\in[\lceil\beta-s_0-1\rceil,\beta- s_0);$ 
\item\label{a2iii}$\norm{\C_0^{\frac{s}2}\C_1^{-\frac12}u}\leq c_2\norm{\C_0^{\frac{s-\beta}2}u}, \quad\forall u\in X^{\beta-s}, \;\forall s\in(s_0,1];$
\item\label{a2iv}$\norm{\C_0^{-\frac{s}2}\C_1^{-\frac12}\A^{-1}u}\leq c_3\norm{\C_0^{\frac{2\ell-\beta-s}2}u}, \quad\forall u\in X^{s+\beta-2\ell}, \;\forall s\in(s_0,1];$

\item\label{a2v}$\norm{\C_0^{\frac{\eta}2}\A^{-1}\C_1^{-1}u}\leq c_4\norm{\C_0^{\frac{\eta}2+\ell-\beta}u}, \quad\forall u\in X^{2\beta-2\ell-\eta}, \;\forall \eta\in [\beta-2\ell, 1].$
\end{enumerate}
\end{ass}

Notice that, by Assumption \ref{a2}(\ref{a2delta}) we have $2\ell-\beta>-1$ which, in combination with Assumption \ref{a2}(\ref{a2i}), implies that \begin{equation}\pr{\C_1^{-\frac12}\A^{-1}u}{\C_1^{-\frac12}\A^{-1}u}+\lambda\pr{\C_0^{-\frac12}u}{\C_0^{-\frac12}u}\leq c\pr{\C_0^{-\frac12}u}{\C_0^{-\frac12}u}, \quad\forall u\in X^1,\nonumber\end{equation} capturing the idea that the regularization through $\C_0$ is indeed a regularization. In fact the assumption $\Delta>2s_0$ connects the ill-posedness of the problem to the regularity of the prior. We exhibit this connection in the following example:
\begin{ex} Assume $\A, \C_1$ and $\C_0$ are simultaneously diagonalizable, with eigenvalues having algebraic decay $k^{2t}$, $k^{-2r}$ and $k^{-2\alpha}$, respectively, for $t,r\geq0$ and $\alpha>\frac12$ so that  $\C_0$ is trace class. Then Assumptions (\ref{a1}),(\ref{a2i})-(\ref{a2v}) are trivially satisfied with $\ell=\frac{t}{\alpha},$ $\beta=\frac{r}{\alpha}$ and  $s_0=\frac{1}{2\alpha}$. The Assumption (\ref{a2delta}) $\Delta>2s_0$ is then equivalent to $\alpha>1+r-2t$. That is, for a certain degree of ill-posedness (encoded in the difference $2t-r$) we have a minimum requirement on the regularity of the prior (encoded in $\alpha$). 
Put differently, for a certain prior, we require a minimum degree of ill-posedness.
\end{ex}
We refer the reader to Section \ref{sec:ex} for nontrivial examples satisfying Assumptions \ref{a2}.
 
In the following, we exploit the regularity properties of a white noise to determine the regularity of draws from the prior and the noise distributions using Assumption \ref{a2}(\ref{a1}). We consider a white noise to be a draw from $\G(0,I)$, that is a random variable $\zeta\sim\G(0,I)$. Even though the identity operator is not trace class in $\X$, it is trace class in a bigger space $X^{-s}$, where $s>0$ is sufficiently large.

\begin{lem}\label{l1}Under the Assumption \ref{a2}(\ref{a1}) we have:
\begin{enumerate}
\item[i)]Let $\zeta$ be a white noise. Then $\E\norm{\C_0^{\frac{s}2}\zeta}^2<\infty$ for all $s>s_0$.
\item[ii)]Let $u\sim\mu_0$. Then $u\in X^{1-s}$ $\;\mu_0$-a.s. for every $s>s_0$.
\end{enumerate}
\end{lem}
\begin{proof}{\ }
\begin{enumerate}
\item[i)]We have that $\C_0^\frac{s}2\zeta\sim\G(0,\C_0^s)$, thus $\E\norm{\C_0^\frac{s}2\zeta}^2<\infty$ is equivalent to $\C_0^s$ being of trace class. By the Assumption $\ref{a2}(\ref{a1})$ it suffices to have $s>s_0$.
\item[ii)]We have $\E\norm{\C_0^{\frac{s-1}2}u}^2=\E\norm{\C_0^{\frac{s}2}\C_0^{-\frac12}u}^2=\E\norm{\C_0^{\frac{s}2}\zeta}^2$, where $\zeta$ is a white noise, therefore using part (i) we get the result.
\end{enumerate}
\end{proof}

\begin{rem}\label{remark1}
Note that as $s_0$ changes, both the Hilbert scale and the decay of the coefficients of a draw from $\mu_0$ change. The norms $\|{\cdot}\|_t$ are defined through powers of the eigenvalues $\lambda_k^2$. If $s_0>0$, then $\C_0$ has eigenvalues that decay like $k^{-\frac{1}{s_0}}$, thus an element $u\in X^t$ has coefficients $\pr{u}{\phi_k}$, that decay faster than $k^{-\frac12-\frac{t}{2s_0}}$. 
As $s_0$ gets closer to zero, the space $X^t$ for fixed $t>0$, corresponds to a faster decay rate of the coefficients. At the same time, by the last lemma, draws from $\mu_0=\G(0,\C_0)$ belong to $X^{1-s}$ for all $s>s_0$. Consequently, as $s_0$ gets smaller, not only do draws from $\mu_0$ belong to $X^{1-s}$ for smaller $s$, but also the spaces $X^{1-s}$ for fixed $s$ reflect faster decay rates of the coefficients.  The case $s_0=0$ corresponds to $\C_0$ having eigenvalues that decay faster than any negative power of $k$. A draw from $\mu_0$ in that case has coefficients that decay faster than any negative  power of $k$. 
\end{rem}

In the next lemma, we use the interrelations between the operators $\C_0, \C_1, \A^{-1}$ to obtain additional regularity properties of draws from the prior, and also determine the regularity of draws from the noise distribution and the joint distribution of the unknown and the data.
\begin{lem}\label{l2}
Under the Assumptions \ref{a2} we have:
\begin{enumerate}
\item[i)]$u\in X^{s_0+\beta-2\ell+\varepsilon} \;\;\mu_0$-a.s. for all $0<\varepsilon<(\Delta-2 s_0)\wedge(1-s_0);$
\item[ii)]$\A^{-1}u\in\D(\C_1^{-\frac12}) \;\;\mu_0$-a.s.$;$ 
\item[iii)]$\xi\in X^\rho \;\;\rp_0$-a.s. for all $\rho<\beta- s_0;$
\item[iv)]$y\in X^{\rho}$ $\nu$-a.s. for all $\rho<\beta-s_0$.
\end{enumerate}
\end{lem}

\begin{proof}{\ }\begin{enumerate}
\item[i)]We can choose an $\varepsilon$ as in the statement by the Assumption \ref{a2}(\ref{a2delta}).
By Lemma \ref{l1}(ii), it suffices to show that $s_0+\beta-2\ell+\varepsilon<1-s_0$. Indeed, $s_0+\beta-2\ell+\varepsilon=s_0+1-\Delta+\varepsilon<1-s_0.$
\item[ii)]Under Assumption \ref{a2}(\ref{a2i}) it suffices to show that $u\in X^{\beta-2\ell}$. Indeed, by Lemma \ref{l1}(ii), we need to show that $\beta-2\ell<1-s_0$, which is true since $s_0\in[0,1)$ and we assume $\Delta>2 s_0\geq s_0$, thus $2\ell-\beta+1>s_0.$
\item[iii)]It suffices to show it for any $\rho\in[\lceil\beta-s_0-1\rceil,\beta-s_0)$. Noting that $\zeta=\C_1^{-\frac12}\xi$ is a white noise, using Assumption \ref{a2}(\ref{a2ii}), we have by Lemma \ref{l1}(i)
\begin{equation}\E{\|\xi\|}^2_{\rho}=\E\norm{\C_0^{-\frac{\rho}2}\C_1^{\frac12}\C_1^{-\frac12}\xi}^2\leq c\E\norm{\C_0^{\frac{\beta-\rho}2}\zeta}^2<\infty,\nonumber\end{equation}
 since $\beta-\rho>s_0$. 
\item[iv)]By (ii) we have that $\A^{-1}u$ is $\mu_0$-a.s. in the Cameron-Martin space of the Gaussian measures $\rp$ and $\rp_0$, thus the measures $\rp$ and $\rp_0$ are $\mu_0$-a.s. equivalent~\cite[Theorem 2.8]{MR2244975} and (iii) gives the result.\end{enumerate}
\end{proof}

\subsection{Guidelines for applying the theory}

The theory is naturally developed in the scale of Hilbert
spaces defined via the prior. However application of
the theory may be more natural in a different functional
setting. We explain how the two may be connected.
Let $\{\psi_k\}_{k\in\N}$ be an orthonormal basis of the separable Hilbert space $\X$. We define the spaces $\hat{X}^t, \; t\in\R$ as follows: for $t>0$ we set 
\begin{equation}\hat{X}^t:=\{u\in\X:\; \sumi k^{2t}\pr{u}{\psi_k}^2<\infty\}\nonumber\end{equation} and the spaces $\hat{X}^{-t}, \;t>0$ are defined by duality, $\hat{X}^{-t}:=(\hat{X}^{t})^\ast.$ 

For example, if we restrict ourselves to functions on a periodic domain $D=[0,L]^d$ and assume that  $\{\psi_k\}_{k\in\N}$ is the Fourier basis of $\X=L^2(D)$, then the spaces $\hat{X}^t$ can be identified with the Sobolev spaces of periodic functions $H^t$, by rescaling: $H^t=\hat{X}^{\frac{t}{d}}$~\cite[Proposition 5.39]{MR1881888}. 

In the case $s_0>0$, as explained in Remark \ref{remark1} we have algebraic decay of the eigenvalues of $\C_0$ and in particular $\lambda_k^2$ decay like $k^{-\frac{1}{s_0}}$. If $\C_0$ is diagonalizable in the basis $\{\psi_k\}_{k\in\N}$, that is, if $\phi_k=\psi_k, \;k\in\N$, then it is straightforward to identify the spaces $X^t$ with the spaces $\hat{X}^{\frac{t}{2s_0}}$.  The advantage of this identification is that the spaces $\hat{X}^{t}$ do not depend on the prior so one can use them as a fixed reference point for expressing regularity, for example of the true solution. 

In our subsequent analysis, we will require that the true solution lives in the Cameron-Martin space of the prior $X^1$, which in different choices of the prior (different $s_0$) is a different space. Furthermore, we will assume that the true solution lives in $X^{\gamma}$ for some $\gamma\geq1$ and provide the convergence rate depending on the parameters $\gamma, s_0, \beta, \ell$.  The identification $X^\gamma=\hat{X}^{\frac{\gamma}{2s_0}}$ and the intuitive relation between the spaces $\hat{X}^t$ and the Sobolev spaces, enable us to understand the meaning of the assumptions on the true solution.

We can now formulate the following guidelines for applying the theory presented in the present paper: we work in a separable Hilbert space $\X$ with an orthonormal basis $\{\psi_k\}_{k\in\N}$ and we have some prior knowledge about the true solution $u^\dagger$ which can be expressed in terms of the spaces $\hat{X}^t$. The noise is assumed to be Gaussian $\G(0,\C_1)$, and the forward operator is known; that is, $\C_1$ and $\A^{-1}$ are known. We choose the prior $\G(0,\C_0)$, that is, we choose the covariance operator $\C_0$, and we can determine the value of $s_0$. If the operator $\C_0$ is chosen to be diagonal in the basis $\{\psi_k\}_{k\in\N}$ then we can find the regularity of the true solution in terms of the spaces $X^t$, that is, the value of $\gamma$ such that $u^\dagger\in X^\gamma$, and check that $\gamma\geq1$ which is necessary for our theory to work. We then find the values of $\beta$ and $\ell$ and calculate the value of $\Delta$ appearing in Assumption \ref{a2},  checking that our choice of the prior is such that $\Delta>2s_0$. We now have all the necessary information required for applying the Theorems \ref{pdecor1} and \ref{pdecor2} presented in Section \ref{sec:mainresults} to get the rate of convergence. 

\begin{rem}
Observe that in the above mentioned example of periodic functions, we have the identification $X^1=H^{\frac{d}{2s_0}}$, thus since $s_0<1$ we have that the assumption $u^\dagger\in X^1$ implies that $u^\dagger\in H^t$, for $t>\frac{d}2$. By the Sobolev embedding theorem~\cite[Theorem 5.31]{MR1881888}, this implies that the true solution is always assumed to be continuous. However, this is not a disadvantage of our method, since in many cases a Gaussian measure which charges $L^2(D)$ with probability one, can be shown to also charge the space of continuous functions with probability one~\cite[Lemma 6.25]{MR2652785} 
\end{rem}
%%%%%%%%%%%%%% Section Properties of the Posterior Mean and Covariance %%%%%%%%%%%%%

\section{Properties of the Posterior Mean and Covariance}\label{sec:prmean}
We now make sense of the equation $(\ref{eq:int7})$ weakly in the space $X^1$, under the assumptions presented in the previous section. To do so, we define the operator $\B_\lambda$ from (\ref{eq:int6}) in $X^1$ and examine its properties. In Section \ref{sec:justification} we demonstrate that (\ref{eq:int4}) and (\ref{eq:int7}) do indeed correspond to the posterior covariance and mean. 

Consider the equation \begin{equation}\label{eq:pm1}\B_\lambda w= r, \end{equation} where \begin{equation}\B_\lambda=\A^{-1}\C_1^{-1}\A^{-1}+\lambda \C_0^{-1}.\nonumber\end{equation}
Define the bilinear form $B\colon  X^1 \times X^1\to \R$, \begin{equation}B(u,v):=\pr{\C_1^{-\frac12}\A^{-1} u}{\C_1^{-\frac12}\A^{-1}v}+\lambda\pr{\C_0^{-\frac12}u}{\C_0^{-\frac12}v}, \;\;\forall u,v\in X^1.\nonumber\end{equation}
\begin{defn}\label{weak}
Let $r\in X^{-1}$. An element $w\in X^1$ is called a weak solution of (\ref{eq:pm1}), if  \begin{equation}B(w,v)=\pr{r}{v}, \;\forall v\in X^1.\nonumber\end{equation}
\end{defn}

\begin{prop}\label{laxmil}
Under the Assumptions \ref{a2}(\ref{a2delta}) and (\ref{a2i}), for any $r\in X^{-1}$, there exists a unique weak solution $w\in X^1$ of (\ref{eq:pm1}).
\end{prop}

\begin{proof}
We use the Lax-Milgram theorem in the Hilbert space $X^1$, since $r\in X^{-1}=(X^1)^\ast.$ 
\begin{enumerate}
\item[i)]$B\colon X^1\times X^1\to\R$ is coercive: \begin{equation}B(u,u)=\norm{\C_1^{-\frac12}\A^{-1}u}^2+\lambda\norm{\C_0^{-\frac12}u}^2\geq\lambda{\|u\|}_{1}^2, \;\forall u\in X^1.\nonumber\end{equation}
\item[ii)]$B\colon X^1\times X^1\to\R$ is continuous: indeed  by the Cauchy-Schwarz inequality and the Assumptions \ref{a2}(\ref{a2delta}) and (\ref{a2i}),   
\begin{align*}\vert B(u,v)\vert&\leq \norm{\C_1^{-\frac12}\A^{-1}u}\norm{\C_1^{-\frac12}\A^{-1}v}+\lambda\norm{\C_0^{-\frac12}u}\norm{\C_0^{-\frac12}v}\\
&\leq c{\|u\|}_{{\beta-2\ell}}\norm{v}_{{\beta-2\ell}}+\lambda{\|u\|}_{1}\norm{v}_{1}\leq c'{\|u\|}_{1}\norm{v}_{1}, \;\forall u,v\in X^1.\end{align*} 
\end{enumerate}
\end{proof}

\begin{rem}\label{yosida}
The Lax-Milgram theorem defines a bounded operator $\Sc:X^{-1}\to X^1$, such that $B(\Sc r,v)=\pr{r}{v}$ for all $v\in X^1$, which has a bounded inverse $\Sc^{-1}:X^1\to X^{-1}$ such that $B(w,v)=\pr{\Sc^{-1}w}{v}$ for all $v\in X^1$. Henceforward, we identify $\B_\lambda\equiv \Sc^{-1}$ and $\B_\lambda^{-1}\equiv \Sc$.
Furthermore, note that in Proposition \ref{laxmil}, Lemma \ref{id} below, and the three propositions in Section \ref{sec:normbounds}, we
only require $\Delta>0$ and not the stronger assumption $\Delta>2s_0$. However, in all our other results we actually need $\Delta>2s_0.$

\end{rem}

\begin{lem}\label{id}
Suppose the Assumptions \ref{a2}(\ref{a2delta}) and (\ref{a2i}) hold. Then the operator $\Sc^{-1}=\B_\lambda:X^{1}\to X^{-1}$ is identical to the operator $\A^{-1}\C_1^{-1}\A^{-1}+\lambda \C_0^{-1}: X^{1}\to X^{-1}$, where $\A^{-1}\C_1^{-1}\A^{-1}$ is defined weakly in $X^{\beta-2\ell}$.\end{lem}
\begin{proof}
The Lax-Milgram theorem implies that $\B_\lambda:X^1\to X^{-1}$ is bounded. Moreover, $\C_0^{-1}:X^1\to X^{-1}$ is bounded, thus the operator $K\coloneq \B_\lambda- \lambda\C_0^{-1}: X^1\to X^{-1}$ is also bounded and satisfies \begin{equation}\label{star1}\pr{Ku}{v}=\pr{\C_1^{-\frac12}\A^{-1}u}{\C_1^{-\frac12}\A^{-1}v}, \;\forall u,v \in X^1.\end{equation}
Define $\A^{-1}\C_1^{-1}\A^{-1}$ weakly in $X^{\beta-2\ell}$, by the bilinear form $A:X^{\beta-2\ell}\times X^{\beta-2\ell}\to\R$ given by \begin{equation}A(u,v)=\pr{\C_1^{-\frac12}\A^{-1}u}{\C_1^{-\frac12}\A^{-1}v}, \;\forall u,v \in X^{\beta-2\ell}.\nonumber\end{equation} By Assumption \ref{a2}(\ref{a2i}), $A$ is coercive and continuous in $X^{\beta-2\ell}$, thus by the Lax-Milgram theorem, there exists a uniquely defined, boundedly invertible, operator $T:X^{2\ell-\beta}\to X^{\beta-2\ell}$ such that $A(u,v)=\pr{T^{-1}u}{v}$ for all $v\in X^{\beta-2\ell}$. We identify $\A^{-1}\C_1^{-1}\A^{-1}$ with the bounded operator $T^{-1}:X^{\beta-2\ell}\to X^{2\ell-\beta}$. By Assumption \ref{a2}(\ref{a2delta}) we have $\Delta>0$ hence \begin{equation}\norm{\A^{-1}\C_1^{-1}\A^{-1}u}_{-1}\leq c\norm{\A^{-1}\C_1^{-1}\A^{-1}u}_{2\ell-\beta}\leq c\norm{u}_{\beta-2\ell}\leq c\norm{u}_1, \;\forall u\in X^1,\nonumber\end{equation} that is, $\A^{-1}\C_1^{-1}\A^{-1}: X^{1}\to X^{-1}$ is bounded. By the definition of $T^{-1}=\A^{-1}\C_{1}^{-1}\A^{-1}$ and $(\ref{star1})$, this implies that $K=\B_\lambda-\lambda\C_0^{-1}=\A^{-1}\C_1^{-1}\A^{-1}$.
\end{proof}

\begin{prop}\label{meanlem}
Under the Assumptions \ref{a2}(\ref{a1}),(\ref{a2delta}),(\ref{a2i}),(\ref{a2ii}),(\ref{a2v}), there exists a unique weak solution, $m\in X^1$ of equation (\ref{eq:int7}), $\nu(du,dy)$-almost surely.
\end{prop}
\begin{proof}
It suffices to show that $\A^{-1}\C_1^{-1}y\in X^{-1},$ $\nu(du,dy)$-almost surely. Indeed, by Lemma \ref{l2}(iv) we have that $y\in X^\rho$ $\nu(du,dy)$-a.s. for all $\rho<\beta-s_0$, thus by the Assumption \ref{a2}(\ref{a2v}) \begin{equation}\norm{\C_0^{\frac12}\A^{-1}\C_1^{-1}y}\leq c\norm{\C_0^{\frac12+\ell-\beta}y}<\infty,\nonumber\end{equation} since $2\beta-2\ell-1<\beta-s_0$, which holds by the Assumption \ref{a2}(\ref{a2delta}).
\end{proof}

%%%%%%%%%%%% Section Posterior Identification %%%%%%%%%%%%%%%

\section{Characterization of the Posterior using Precision Operators}\label{sec:justification}

Suppose that in the problem (\ref{eq:int1}) we have $u\sim\mu_0=\G(0, \C_0)$ and  $\xi\sim\G(0,\C_1)$, where $u$ is independent of $\xi$. Then we have that $y|u\sim\rp=\G(\A^{-1}u, \frac1n\C_1)$. Let $\mu^y$ be the posterior measure on $u|y$. 
 
In this section we prove a number of facts concerning the posterior measure $\mu^y$ for $u|y$. First, in Theorem \ref{prop0} we prove that this measure has density with respect to the prior measure $\mu_0$, identify this density and show that $\mu^y$ is Lipschitz in $y$, with respect to the Hellinger metric. Continuity in $y$ will require the introduction of the space $X^{s+\beta-2\ell}$, to which $u$ drawn from $\mu_0$ belongs almost surely. Secondly, we prove Theorem \ref{justth}, where we show that $\mu^y$ is Gaussian and identify the covariance and mean via equations (\ref{eq:int4}) and (\ref{eq:int7}). This identification will form the basis for our analysis of posterior contraction in the following section. 
 
 \begin{thm}\label{prop0}Under the Assumptions \ref{a2}(\ref{a1}),(\ref{a2delta}),(\ref{a2i}),(\ref{a2ii}),(\ref{a2iii}),(\ref{a2iv}), the posterior measure $\mu^y$ is absolutely continuous with respect to $\mu_0$ and 
\begin{equation}\label{eq:rn}\frac{d\mu^y}{d\mu_0}(u)=\frac{1}{Z(y)}\exp(-\Phi(u,y)), \end{equation}
where \begin{equation}\label{eq:Phi}\Phi(u,y)\coloneq \frac{n}2\norm{\C_1^{-\frac12}\A^{-1}u}^2-n\pr{\C_1^{-\frac12}y}{\C_1^{-\frac12}\A^{-1}u}\end{equation} and $Z(y)\in(0,\infty)$ is the normalizing constant.
Furthermore, the map $y\mapsto\mu^y$ is Lipschitz continuous, with respect to the Hellinger metric: let $s=s_0+\varepsilon$, $\;0<\varepsilon<(\Delta-2 s_0)\wedge(1-s_0)$; then there exists $c=c(r)$ such that for all $y, y'\in X^{\beta-s}$ with ${\|y\|}_{{\beta-s}}, \norm{y'}_{{\beta-s}}\leq r$
\begin{equation}d_{\mathrm{Hell}}(\mu^y,\mu^{y'})\leq c\norm{y-y'}_{{\beta-s}}.\nonumber\end{equation}
Consequently, the $\mu^y$-expectation of any polynomially bounded function \\ $f\colon X^{s+\beta-2\ell}\to E,$ where $(E,\|\cdot\|_{E})$ is a Banach space, is locally Lipschitz continuous in $y$. In particular, the posterior mean is locally Lipschitz continuous in $y$ as a function $X^{\beta-s}\to X^{s+\beta-2\ell}$.\end{thm}

The proofs of Theorem \ref{prop0} and Theorem \ref{justth} are presented in the next two subsections. Each proof is based on a series of lemmas. 

\subsection{Proof of Theorem \ref{prop0}}
In this subsection we prove Theorem \ref{prop0}. We first prove several useful estimates regarding $\Phi$ defined in (\ref{eq:Phi}), for $u\in X^{s+\beta-2\ell}$ and $y\in X^{\beta-s}$, where $s\in(s_0,1]$. Observe that, under the Assumptions \ref{a2}(\ref{a1}),(\ref{a2delta}),(\ref{a2i}),(\ref{a2ii}), for $s=s_0+\varepsilon$ where $\varepsilon>0$ sufficiently small, the Lemma \ref{l2} implies on the one hand that $u\in X^{s+\beta-2\ell}$ $\mu_0(du)$-almost surely and on the other hand that $y\in X^{\beta-s}$ $\nu(du,dy)$-almost surely.

\begin{lem}\label{asslem}
Under the Assumptions \ref{a2}(\ref{a1}),(\ref{a2i}),(\ref{a2iii}),(\ref{a2iv}), for any $s\in(s_0,1]$, the potential $\Phi$ given by (\ref{eq:Phi}) satisfies:
 
\begin{enumerate}
\item[i)]for every $\delta>0$ and $r>0,$ there exists an $M=M(\delta,r)\in\R,$ such that for  all $u\in X^{s+\beta-2\ell}$ and all $y\in X^{\beta-s}$ with ${\|y\|}_{{\beta-s}}\leq r,$  \begin{equation}\Phi(u,y)\geq M-\delta{\|u\|}_{{s+\beta-2\ell}}^2;\nonumber\end{equation}
\item[ii)]for every $r>0,$ there exists a $K=K(r)>0,$ such that for all $u\in X^{s+\beta-2\ell}$ and $y\in X^{\beta-s}$ with ${\|u\|}_{{s+\beta-2\ell}},{\|y\|}_{{\beta-s}}\leq r,$ \begin{equation}\Phi(u,y)\leq K;\nonumber\end{equation}
\item[iii)]for every $r>0,$ there exists an $L=L(r)>0,$ such that for all $u_1, u_2\in X^{s+\beta-2\ell}$ and $y\in X^{\beta-s}$  with ${\|u_1\|}_{{s+\beta-2\ell}}, {\|u_2\|}_{{s+\beta-2\ell}}, {\|y\|}_{{\beta-s}}\leq r,$ \begin{equation}|\Phi(u_1,y)-\Phi(u_2,y)|\leq L{\|u_1-u_2\|}_{{s+\beta-2\ell}};\nonumber\end{equation}
\item[iv)]for every $\delta>0$ and $r>0,$ there exists a $c=c(\delta,r)\in\R,$ such that for all $y_1, y_2\in X^{\beta-s}$ with ${\|y_1\|}_{{\beta-s}}, {\|y_2\|}_{{\beta-s}}\leq r$ and for all $u\in X^{s+\beta-2\ell},$ 
\begin{equation}\vert\Phi(u,y_1)-\Phi(u,y_2)\vert\leq \exp\left(\delta{\|u\|}^2_{{s+\beta-2\ell}}+c\right){\|y_1-y_2\|}_{{\beta-s}}.\nonumber\end{equation} 
\end{enumerate}
\end{lem}

\begin{proof}${\;}$
\begin{enumerate}
\item[i)]
By first using the Cauchy-Schwarz inequality, then the Assumptions \ref{a2} (\ref{a2iii}) and (\ref{a2iv}), and then the Cauchy with $\delta'$ inequality for $\delta'>0$ sufficiently small, we have  \begin{align*}\Phi(u,y)&=\frac{n}2\norm{\C_1^{-\frac12}\A^{-1}u}^2-n\pr{\C_0^{\frac{s}2}\C_1^{-\frac12}y}{\C_0^{-\frac{s}2}\C_1^{-\frac12}\A^{-1}u}\\
&\geq-n\norm{\C_0^{\frac{s}2}\C_1^{-\frac12}y}\norm{\C_0^{-\frac{s}2}\C_1^{-\frac12}\A^{-1}u}\geq-cn{\|y\|}_{{\beta-s}}{\|u\|}_{{s+\beta-2\ell}}\\&\geq -\frac{cn}{4\delta'}{\|y\|}_{{\beta-s}}^2-cn\delta'{\|u\|}^2_{{s+\beta-2\ell}}\geq M(r,\delta)-\delta{\|u\|}^2_{{s+\beta-2\ell}}.\end{align*}

\item[ii)]
 By the Cauchy-Schwarz  inequality and the Assumptions \ref{a2}(\ref{a2i}),(\ref{a2iii}) and (\ref{a2iv}), we have since $s>s_0\geq0$ \begin{align*}\Phi(u,y)
&\leq\frac{n}2\norm{\C_1^{-\frac12}\A^{-1}u}^2+n\norm{\C_0^{\frac{s}2}\C_1^{-\frac12}y}\norm{\C_0^{-\frac{s}2}\C_1^{-\frac12}\A^{-1}u}\\
&\leq c\frac{n}2{\|u\|}^2_{{\beta-2\ell}}+cn{\|y\|}_{{\beta-s}}{\|u\|}_{{s+\beta-2\ell}}\leq K(r).\end{align*} 

\item[iii)]
 By first using the Assumptions \ref{a2} (\ref{a2iii}) and (\ref{a2iv}) and the triangle inequality, and then the Assumption \ref{a2}(\ref{a2i}) and the reverse triangle inequality, we have since $s>s_0\geq0$
 \begin{equation}|\Phi(u_1,y)-\Phi(u_2,y)|=\nonumber\end{equation}\begin{equation}\frac{n}2\left|\norm{\C_1^{-\frac12}\A^{-1}u_1}^2-\norm{\C_1^{-\frac12}\A^{-1}u_2}^2+2\pr{\C_0^{\frac{s}2}\C_1^{-\frac12}y}{\C_0^{-\frac{s}2}\C_1^{-\frac12}\A^{-1}(u_2-u_1)}\right|\nonumber\end{equation}\begin{equation}\leq\frac{n}2\left|\norm{\C_1^{-\frac12}\A^{-1}u_1}^2-\norm{\C_1^{-\frac12}\A^{-1}u_2}^2\right|+cn{\|y\|}_{{\beta-s}}{\|u_1-u_2\|}_{{s+\beta-2\ell}}\nonumber\end{equation}\begin{equation}\leq cn{\|u_1-u_2\|}_{{\beta-2\ell}}\left({\|u_1\|}_{{\beta-2\ell}}+{\|u_2\|}_{{\beta-2\ell}}\right)+cnr{\|u_1-u_2\|}_{{s+\beta-2\ell}}\nonumber\end{equation}\begin{equation}\leq L(r){\|u_1-u_2\|}_{{s+\beta-2\ell}}.\nonumber\end{equation}
\item[iv)]
By first using the Cauchy-Schwarz inequality and then the Assumptions \ref{a2}(\ref{a2iii}) and (\ref{a2iv}), we have \begin{align*}|\Phi(u,y_1)-\Phi(u,y_2)|&=n\left|\pr{\C_0^{\frac{s}2}\C_1^{-\frac12}(y_1-y_2)}{\C_0^{-\frac{s}2}\C_1^{-\frac12}\A^{-1}u}\right|\\&\leq n\norm{\C_0^{\frac{s}2}\C_1^{-\frac12}(y_1-y_2)}\norm{\C_0^{-\frac{s}2}\C_1^{-\frac12}\A^{-1}u}\\&\leq cn{\|y_1-y_2\|}_{{\beta-s}}{\|u\|}_{{s+\beta-2\ell}}\\&\leq\exp\left(\delta\norm{u}^2_{s+\beta-2\ell}+c\right)\norm{y_1-y_2}_{\beta-s}.\end{align*}

\end{enumerate}
\end{proof}

\begin{cor}\label{c1}
Under the Assumptions \ref{a2}(\ref{a1}),(\ref{a2delta}),(\ref{a2i}),(\ref{a2iii}),(\ref{a2iv})\begin{equation}Z(y)\coloneq \int_{\X}\exp(-\Phi(u,y))\mu_0(du)>0,\nonumber\end{equation} for all $y\in X^{\beta-s}, s=s_0+\varepsilon$ where $0<\varepsilon<(\Delta-2s_0)\wedge(1-s_0)$. In particular, if in addition the Assumption \ref{a2}(\ref{a2ii}) holds, then $Z(y)>0$ $\nu$-almost surely.
\end{cor}
\begin{proof}
Fix $y\in X^{\beta-s}$ and set $r={\|y\|}_{{\beta-s}}$.
Gaussian measures on  separable Hilbert spaces are full~\cite[Proposition 1.25]{MR2244975}, hence since by Lemma \ref{l2}(i) $\mu_0(X^{s+\beta-2\ell})=1,$ we have that $\mu_0(B_{X^{s+\beta-2\ell}}(r))>0$. By Lemma \ref{asslem}(ii), there exists $K(r)>0$ such that
\begin{align*}\int_{\X}\exp(-\Phi(u,y))\mu_0(du)&\geq\int_{B_{X^{s+\beta-2\ell}}(r)}\exp(-\Phi(u,y))\mu_0(du)\\&
\geq\int_{B_{X^{s+\beta-2\ell}}(r)}\exp(-K(r))\mu_0(du)>0.\end{align*} Recalling that, under the additional Assumption \ref{a2}(\ref{a2ii}), by Lemma \ref{l2}(iv) we have $y\in X^{\beta-s}$ $\nu$-almost surely for all $s>s_0$, completes the proof. 
\end{proof}

We are now ready to prove Theorem \ref{prop0}:

\begin{proof}[Proof of Theorem \ref{prop0}]
Recall that $\nu_0=\rp_0(dy)\otimes\mu_0(du)$ and $\nu=\rp(dy|u)\mu_0(du).$ By the Cameron-Martin formula~\cite[Corollary 2.4.3]{MR1642391}, since by Lemma \ref{l2}(ii) we have $\A^{-1}u\in\D(\C_1^{-\frac12})\; \mu_0$-a.s., we get for $\mu_0$-almost all $u$ \begin{equation}\frac{d\rp}{d\rp_0}(y|u)=\exp(-\Phi(u,y)),\nonumber\end{equation} thus we have for $\mu_0$-almost all $u$ \begin{equation}\frac{d\nu}{d\nu_0}(y,u)=\exp(-\Phi(u,y)).\nonumber\end{equation}
By~\cite[Lemma 5.3]{MR2358638} and Corollary \ref{c1} we have the relation (\ref{eq:rn}).\\
For the proof of the Lipschitz continuity of the posterior measure in $y$, with respect to the Hellinger distance, we apply~\cite[Theorem 4.2]{MR2652785} for $Y=X^{\beta-s}, X=X^{s+\beta-2\ell}$, using Lemma \ref{asslem} and the fact that $\mu_0(X^{s+\beta-2\ell})=1$, by Lemma \ref{l2}(i).
\end{proof}

\subsection{Proof of Theorem \ref{justth}}
We first give an overview of the proof of Theorem \ref{justth}. 
Let $y|u\sim\rp=\G(\A^{-1}u,\frac1n\C_1)$ and $u\sim\mu_0$. Then by Proposition \ref{meanlem},  there exists a unique weak solution, $m\in X^1$, of (\ref{eq:int7}), $\nu(du,dy)$-almost surely. That is, with $\nu(du,dy)$-probability equal to one, there exists an $m=m(y)\in X^1$ such that  \begin{equation}B(m,v)=b^y(v), \;\;\forall v\in X^1,\nonumber\end{equation} where the bilinear form $B$ is defined in Section \ref{sec:prmean}, and $b^y(v)=\pr{\A^{-1}\C_1^{-1}y}{v}$. In the following we show that $\mu^y=\G(m,\C)$, where \begin{equation}\C^{-1}=n\A^{-1}\C_1^{-1}\A^{-1}+\frac{1}{\tau^2}\C_0^{-1}.\nonumber\end{equation}The proof has the same structure as the proof for the identification of the posterior in~\cite{Pokern:2012fk}. We define the Gaussian measure $\G(m^N,\C^N)$, which is the independent product of a measure identical to $\G(m,\C)$ in the finite-dimensional space $\X^N$ spanned by the first $N$ eigenfunctions of $\C_0$, and a measure identical to $\mu_0$ in $(\X^N)^\perp$. We next show that $\G(m^N,\C^N)$ converges weakly to the measure $\mu^y$ which as a weak limit of Gaussian measures has to be Gaussian $\mu^y=\G(\overline{m}, \overline{\C})$, and we then identify $\overline{m}$ and $\overline{\C}$ with $m$, $\C$ respectively. 

Fix $y$ drawn from $\nu$ and let $P^N$  be the orthogonal projection of $\X$ to the finite-dimensional space $\operatorname{span}\{\phi_1,...,\phi_N\}\coloneq \X^N,$ where as in Section \ref{sec:assumptions}, $\{\phi_k\}_{k=1}^\infty$ is an orthonormal eigenbasis of $\C_0$ in $\X$. Let $Q^N=I-P^N$. We define $\mu^{N,y}$ by \begin{equation}\label{eq:rnfinite}\frac{d\mu^{N,y}}{d\mu_0}(u)=\frac{1}{Z^N(y)}\exp(-\Phi^N(u,y))\end{equation}
where $\Phi^N(u,y):=\Phi(P^Nu,y)$ and \begin{equation}Z^N(y):=\int_{\X}\exp(-\Phi^N(u,y))\mu_0(du).\nonumber\end{equation}

\begin{lem}\label{lemproj} We have $\mu^{N,y}=\G(m^N,\C^N)$, where \begin{equation}P^N\C^{-1}P^Nm^N=nP^N\A^{-1}\C_1^{-1}y,\nonumber\end{equation}   \begin{equation}P^N\C^NP^N=P^N\C P^N, \;Q^N\C^NQ^N=\tau^2Q^N\C_0Q^N\nonumber\end{equation} and $P^N\C^NQ^N=Q^N\C^NP^N=0$.
\end{lem}                                    
\begin{proof}
Let $u\in\X^N$. Since $u=P^Nu$ we have by $(\ref{eq:rnfinite})$  \begin{equation}d\mu^{N,y}(P^Nu)\propto \exp\left(-\Phi(P^Nu;y)\right)d\mu_0(P^Nu).\nonumber\end{equation} 
The right hand side is $N$-dimensional Gaussian with density proportional to the exponential of the following expression
\begin{equation}\label{star}-\frac{n}2\norm{\C_1^{-\frac12}\A^{-1}P^Nu}^2+n\pr{\C_1^{-\frac12}y}{\C_1^{-\frac12}\A^{-1}P^Nu}-\frac{1}{2\tau^2}\norm{\C_0^{-\frac12}P^Nu}^2,\end{equation} which by completing the square we can write as \begin{equation}-\frac{1}2\norm{(\tilde{\C}^N)^{-\frac12}(u-\tilde{m}^N)}^2+c(y),\nonumber\end{equation} where $\tilde{\C}^N$ is the covariance matrix and $\tilde{m}^N$ the mean. By equating with expression (\ref{star}), we find that 
 $(\tilde{\C}^N)^{-1}=P^N\C^{-1}P^N$ and $(\tilde{\C}^N)^{-1}\tilde{m}^N=nP^N\A^{-1}\C_1^{-1}y$, thus on $\X^N$ we have that $\mu^{N,y}=\G(\tilde{m}^N,\tilde{\C}^N)$.
On $(\X^N)^\perp$, the Radon-Nikodym derivative in $(\ref{eq:rnfinite})$ is equal to 1, hence $\mu^{N,y}=\mu_0=\G(0,\tau^2\C_0)$.
\end{proof}

\begin{prop}\label{prop2}
Under the Assumptions \ref{a2}(\ref{a1}),(\ref{a2delta}),(\ref{a2i}),(\ref{a2ii}),(\ref{a2iii}),(\ref{a2iv}), for all $y\in X^{\beta-s}$, $s=s_0+\varepsilon$, where $0<\varepsilon<(\Delta-2 s_0)\wedge(1-s_0)$,  the measures $\mu^{N,y}$ converge weakly in $\X$ to $\mu^y$, where $\mu^y$ is defined in Theorem \ref{prop0}. In particular, $\mu^{N,y}$ converge weakly in $\X$ to $\mu^y \;\nu$-almost surely.

\end{prop}
\begin{proof}
Fix $y\in X^{\beta-s}$.
Let $f:\X\to\R$ be continuous and bounded. Then by $(\ref{eq:rn}), (\ref{eq:rnfinite})$ and Lemma \ref{l2}(i), we have that \begin{equation}\int_{\X}f(u)\mu^{N,y}(du)=\frac{1}{Z^N}\int_{X^{s+\beta-2\ell}}f(u)e^{-\Phi^N(u,y)}\mu_0(du)\nonumber\end{equation}
and \begin{equation}\int_{\X}f(u)\mu^{y}(du)=\frac{1}{Z}\int_{X^{s+\beta-2\ell}}f(u)e^{-\Phi(u,y)}\mu_0(du).\nonumber\end{equation}
Let $u\in X^{s+\beta-2\ell}$ and set $r_1=\max\{{\|u\|}_{{s+\beta-2\ell}}, {\|y\|}_{{\beta-s}}\}$ to get, by Lemma \ref{asslem}(iii), that $\Phi^N(u,y)\to\Phi(u,y),$  since $\norm{P^Nu}_{{s+\beta-2\ell}}\leq{\|u\|}_{{s+\beta-2\ell}}\leq r_1$. By Lemma \ref{asslem}(i), for any $\delta>0,$ for $r_2=\norm{y}_{\beta-s}$, there exists $M(\delta,r_2)\in\R$ such that \begin{equation}\left|f(u)e^{-\Phi^N(u,y)}\right|\leq\norm{f}_{\infty}e^{\delta{\|u\|}^2_{s+\beta-2\ell}-M(\delta,r_2)}, \;\forall u\in X^{s+\beta-2\ell},\nonumber\end{equation}
where the right hand side is $\mu_0$-integrable for $\delta$ sufficiently small by the Fernique Theorem~\cite[Theorem 2.8.5]{MR1642391}. Hence, by the Dominated Convergence Theorem, we have that $\int_{\X}f(u)\mu^{N,y}(du)\to\int_{\X}f(u)\mu^y(du)$, as $N\to\infty$, where we get the convergence of the constants $Z^N\to Z$ by choosing $f\equiv 1$. Thus we have $\mu^{N,y}\Rightarrow\mu^{y}$. Recalling, that $y\in X^{\beta-s} \;\nu$-almost surely completes the proof.
\end{proof}

We are now ready to prove Theorem \ref{justth}:
\begin{proof}[Proof of Theorem \ref{justth}]
By Proposition \ref{prop2} we have that $\mu^{N,y}$ converge weakly in $\X$ to the measure $\mu^y$, $\nu$-almost surely. Since by Lemma \ref{lemproj}, the measures $\mu^{N,y}$ are Gaussian, the limiting measure $\mu^y$ is also Gaussian. To see this we argue as follows. The weak convergence of measures implies the pointwise convergence of the Fourier transforms of the measures, thus by Levy's continuity theorem~\cite[Theorem 4.3]{MR1464694} all the one dimensional projections of $\mu^{N,y}$, which are Gaussian, converge weakly to the corresponding one dimensional projections of $\mu^y$. By the fact that the class of Gaussian distributions in $\R$ is closed under weak convergence~\cite[Chapter 4, Exercise 2]{MR1464694}, we get that all the one dimensional projections of the $\mu^y$ are Gaussian, thus $\mu^y$ is a Gaussian measure in $\X$, $\mu^y=\G(\overline{m},\overline{\C})$ for some $\overline{m}\in\X$ and a self-adjoint, positive semi definite, trace class linear operator $\overline{\C}$. 
It suffices to show that $\overline{m}=m$ and $\overline{\C}=\C$.

We use the standard Galerkin method to show that $m^N\to m$ in $\X$. Indeed, since by their definition $m^N$ solve (\ref{eq:int7}) in the $N$-dimensional spaces $\X^N$, for $e=m-m^N$, we have that $B(e,v)=0, \;\forall v\in \X^{N}$. By  the coercivity and the  continuity of $B$ (see Proposition \ref{laxmil})\begin{equation}\norm{e}_1^2\leq cB(e,e)=cB(e,m-z)\leq c\norm{e}_1\norm{m-z}_1, \;\forall z\in\X^N.\nonumber\end{equation} Choose $z=P^Nm$ to obtain \begin{equation}\norm{m-m^N}\leq c\norm{m-P^Nm}_1,\nonumber\end{equation} where as $N\to\infty$ the right hand side converges to zero since $m\in X^1$. On the other hand, by~\cite[Example 3.8.15]{MR1642391}, we have that $m^N\to\overline{m}$ in $\X$, hence we conclude that $\overline{m}=m$, as required.

For the identification of the covariance operator, note that by the definition of $\C^N$ we have  \begin{equation}\C^N=P^N\C P^N+(I-P^N)\C_0(I-P^N).\nonumber\end{equation} Recall that $\{\phi_k\}_{k=1}^{\infty}$ are the eigenfunctions of $\C_0$ and fix $k\in\N$. Then, for $N>k$ and any $w\in\X$, we have that \begin{equation}\left|\pr{w}{\C^N\phi_k}-\pr{w}{\C\phi_k}\right|=\left|\pr{w}{(P^N-I)\C\phi_k}\right|\nonumber\end{equation}\begin{equation}\leq\norm{(P^N-I)w}\norm{\C\phi_k},\nonumber\end{equation} where the right hand side converges to zero as $N\to\infty$, since $w\in\X$. This implies that $\C^N\phi_k$ converges to $\C\phi_k$ weakly in $\X$, as $N\to\infty$ and this holds for any $k\in\N$. On the other hand by~\cite[Example 3.8.15]{MR1642391}, we have that $\C^N\phi_k\to\overline{\C}\phi_k$ in $\X$, as $N\to\infty$, for all $k\in\N$. It follows that $\overline{\C}\phi_k=\C\phi_k$, for every $k$ and since $\{\phi_k\}_{k=1}^{\infty}$ is an orthonormal basis of $\X$, we have that $\overline{\C}=\C$.
\end{proof}

%%%%%%%%% Section Operadon norm bounds %%%%%%%%%%%%%

\section{Operator norm bounds on $\B_\lambda^{-1}$}\label{sec:normbounds}
The following propositions contain several operator norm estimates on the inverse of $\B_\lambda$ and related quantities, and in particular estimates on the singular dependence of this operator as $\lambda\to0$. These are the key tools used in Section \ref{sec:main} to obtain posterior contraction results. In all of them we make use of the interpolation inequality in Hilbert scales,~\cite[Proposition 8.19]{MR1408680}.
Recall that we consider $\B_\lambda$ defined on $X^1$, as explained in Remark \ref{yosida}.

\begin{prop}\label{pdelem1}
Let $\eta=(1-\theta)(\beta-2\ell)+\theta$, where $\theta\in[0,1]$. Under the Assumption \ref{a2}(\ref{a2i}) the following operator norm bounds hold: there is $c>0$ independent of $\theta$ such that
\begin{equation}\norm{\B_\lambda^{-1}}_{\mathcal{L}(X^{-\eta},X^{\beta-2\ell})}\leq c \lambda^{-\frac{\theta}{2}}\nonumber\end{equation} and 
\begin{equation}\norm{\B_\lambda^{-1}}_{\mathcal{L}(X^{-\eta},X^1)}\leq c\lambda^{-\frac{\theta+1}{2}}.\nonumber\end{equation}
In particular, if $\beta-2\ell\leq0$, interpolation of the two bounds gives \[\norm{\B_\lambda^{-1}}_{\mathcal{L}(X^{-\eta},\X)}\leq c\lambda^{-\frac{\theta+\theta_0}2},\]
where $\theta_0=\frac{2\ell-\beta}{\Delta}\in[0,1]$.
\end{prop}

\begin{proof}
Let $h\in X^{-\eta}\subset X^{-1}$. Then by Proposition \ref{laxmil} for $r=h$, there exists a unique weak solution of (\ref{eq:pm1}), $z\in X^1$. By Definition \ref{weak}, for $v=z\in X^1$, we get \begin{equation}\norm{\C_1^{-\frac12}\A^{-1} z}^2+\lambda\norm{\C_0^{-\frac12}z}^2=\pr{\C_0^{\frac{\eta}{2}}h}{\C_0^{-\frac{\eta}{2}}z}.\nonumber\end{equation}Using the Assumption \ref{a2}(\ref{a2i}), and the Cauchy-Schwarz inequality, we get 
\begin{equation}\norm{z}_{{\beta-2\ell}}^2+\lambda\norm{z}_{1}^2
\leq c\norm{\C_0^{\frac{\eta}2}h}\norm{z}_{{\eta}}.\nonumber\end{equation}
 We interpolate the norm on $z$ appearing on the right hand side between the norms on $z$ appearing on the left hand side,  then use the Cauchy with $\varepsilon$ inequality, and then Young's inequality for $p=\frac{1}{1-\theta}, q=\frac1{\theta}$, to get successively, for $c>0$ a changing constant \begin{equation}\norm{z}_{{\beta-2\ell}}^2+\lambda\norm{z}_{1}^2
\leq c\norm{\C_0^{\frac{\eta}2}h}\norm{z}_{{\beta-2\ell}}^{1-\theta}\lambda^{-\frac{\theta}2}\left(\lambda^{\frac12}\norm{z}_{1}\right)^{\theta}\nonumber\end{equation}
\begin{equation}\leq\frac{c}{2\varepsilon}\left(\lambda^{-\theta}\norm{\C_0^{\frac{\eta}2}h}^2\right)+\frac{c\varepsilon}2\left(\norm{z}_{{\beta-2\ell}}^{2(1-\theta)}\left(\lambda^{\frac12}\norm{z}_{1}\right)^{2\theta}\right)\nonumber\end{equation}
\begin{equation}\leq\frac{c}{2\varepsilon}\left(\lambda^{-\theta}\norm{\C_0^{\frac{\eta}2}h}^2\right)+\frac{c\varepsilon}2\left((1-\theta)\norm{z}_{{\beta-2\ell}}^2+\theta\lambda\norm{z}_{1}^2\right).\nonumber\end{equation} By choosing $\varepsilon>0$ small enough  we get, for $c>0$ independent of $\theta, \lambda$, \begin{equation}\norm{z}_{{\beta-2\ell}}\leq c\lambda^{-\frac{\theta}2}\norm{\C_0^{\frac{\eta_1}2}h}\quad\text{and}\quad\norm{z}_{1}\leq c\lambda^{-\frac{\theta+1}2}\norm{\C_0^{\frac{\eta}2}h}.\nonumber\end{equation} Replacing $z=\B_\lambda^{-1}h$ gives the result.\end{proof}

\begin{prop}\label{pdelem3}
Let $\eta=(1-\theta)(\beta-2\ell-s)+\theta(1-s)$, where $\theta\in[0,1]$ and $s\in(s_0,1]$, where $s_0\in[0,1)$ as defined in Assumption \ref{a2}(\ref{a1}). Under the Assumptions \ref{a2}(\ref{a2delta}) and (\ref{a2i}), the following norm bounds hold: there is $c>0$ independent of $\theta$ such that\begin{equation}\norm{\C_0^{-\frac{s}2}\B_\lambda^{-1}\C_0^{-\frac{s}2}}_{\mathcal{L}(X^{\beta-2\ell-s},X^{-\eta})}\leq c\lambda^{-\frac{\theta}2}\nonumber\end{equation}and \begin{equation}\norm{\C_0^{-\frac{s}2}\B_\lambda^{-1}\C_0^{-\frac{s}2}}_{\mathcal{L}(X^{1-s},X^{-\eta})}\leq c\lambda^{-\frac{\theta+1}2}.\nonumber\end{equation}  In particular, \begin{equation}\norm{\C_0^{-\frac{s}2}\B_\lambda^{-1}\C_0^{-\frac{s}2}}_{\mathcal{L}(X)}\leq c\lambda^{-\frac{2\ell-\beta+s}{\Delta}}, \quad\forall s\in(\{\beta-2\ell\}\vee s_0,1].\nonumber\end{equation}
\end{prop}
\begin{proof}
Let $h\in X^{-\eta}=X^{(1-\theta)\Delta+s-1}$. Then $h\in X^{s-1},$ since $\Delta>0$, thus $\C_0^{-\frac{s}2}h\in X^{-1}$. By Proposition \ref{laxmil} for $r=\C_0^{-\frac{s}2}h$, there exists a unique weak solution of (\ref{eq:pm1}), $z'\in X^{1}$. Since for $v\in X^{1-s}$ we have that $\C_0^{\frac{s}2}v\in X^1$, we conclude that for any $v\in X^{1-s}$ \begin{equation}\pr{\C_1^{-\frac12}\A^{-1}\C_0^{\frac{s}2}z}{\C_1^{-\frac12}\A^{-1}\C_0^{\frac{s}2}v}+\lambda\pr{\C_0^{\frac{s-1}{2}}z}{\C_0^{\frac{s-1}2}v}=\pr{\C_0^{-\frac{s}2}h}{\C_0^{\frac{s}2}v},\nonumber\end{equation}where $z=\C_0^{-\frac{s}2}z'\in X^{1-s}$. Choosing $v=z\in X^{1-s},$ we get
\begin{equation}\norm{\C_1^{-\frac12}\A^{-1}\C_0^{\frac{s}2}z}^2+\lambda\norm{\C_0^{\frac{s-1}{2}}z}^2=\left\langle h,z\right\rangle.\nonumber\end{equation} 
By the Assumption \ref{a2}(\ref{a2i}) and the Cauchy-Schwarz inequality, we have \begin{equation}\norm{z}_{{\beta-2\ell-s}}^2+\lambda\norm{z}_{{1-s}}^2\leq c\norm{h}_{{-\eta}}\norm{z}_{{\eta}}.\nonumber\end{equation} 

We interpolate the norm of $z$ appearing on the right hand side between the norms of $z$ appearing on the left hand side, 
to get as in the proof of Proposition \ref{pdelem1}, for $c>0$ independent
of $\theta, \lambda$ and $s$  
\begin{equation}\norm{z}_{{\beta-2\ell-s}}\leq c\lambda^{-\frac{\theta}2}\norm{h}_{{-\eta}}\quad\text{and}\quad \norm{z}_{{1-s}}\leq c\lambda^{-\frac{\theta+1}2}\norm{h}_{{-\eta}}.\nonumber\end{equation} Replacing $z=\C_0^{-\frac{s}2}\B_\lambda ^{-1}\C_0^{-\frac{s}2}h$ gives the first two rates.\\ 
For the last claim, note that we can always choose $\{\beta-2\ell\}\vee \{s_0\}<s\leq1$, since $s_0<1$ and $\Delta>0$.  
Using the first two estimates, for $\eta=(1-\theta')(\beta-2\ell-s)+\theta'(1-s)=0$, that is $\theta'=\frac{2\ell-\beta+s}{\Delta}\in[0,1],$ we have that \begin{equation}\norm{\C_0^{-\frac{s}2}\B_\lambda^{-1}\C_0^{-\frac{s}2}}_{\mathcal{L}(X^{\beta-2\ell-s},\X)}\leq c\lambda^{-\frac{\theta'}2}\nonumber\end{equation}and\begin{equation}\norm{\C_0^{-\frac{s}2}\B_\lambda^{-1}\C_0^{-\frac{s}2}}_{\mathcal{L}(X^{1-s},\X)}\leq c\lambda^{-\frac{\theta'+1}2}.\nonumber\end{equation}  Let $u\in\X$. Then, for any $t>0$, we have the decomposition\begin{equation}u=\sumi u_k\phi_k=\sum_{\lambda_k^{-1}\leq t}u_k\phi_k+\sum_{\lambda_k^{-1}>t}u_k\phi_k\eqcolon \underline{u}+\overline{u},\nonumber\end{equation} where $\{\phi_k\}_{k=1}^{\infty}$ are the eigenfunctions of $\C_0$ and $u_k\coloneq\pr{u}{\phi_k}.$ Since $1-s\geq0$ and $\beta-2\ell-s<0$, we have \begin{equation}\norm{\C_0^{-\frac{s}2}\B_\lambda^{-1}\C_0^{-\frac{s}2}u}\leq\norm{\C_0^{-\frac{s}2}\B_\lambda^{-1}\C_0^{-\frac{s}2}\underline{u}}+\norm{\C_0^{-\frac{s}2}\B_\lambda^{-1}\C_0^{-\frac{s}2}\overline{u}}\nonumber\end{equation}\begin{equation}\leq c\lambda^{-\frac{\theta'+1}2}\norm{\underline{u}}_{1-s}+c\lambda^{-\frac{\theta'}2}\norm{\overline{u}}_{\beta-2\ell-s}\nonumber\end{equation}
\begin{equation}= c\lambda^{-\frac{\theta'+1}2}\left(\sum_{\lambda_k^{-1}\leq t}\lambda_k^{2s-2}u_k^2\right)^{\frac12}+c\lambda^{-\frac{\theta'}2}\left(\sum_{\lambda_k^{-1}> t}\lambda_k^{2s+4\ell-2\beta}u_k^2\right)^{\frac12}\nonumber\end{equation}\begin{equation}\leq c\lambda^{-\frac{\theta'+1}2}t^{1-s}\norm{u}+c\lambda^{-\frac{\eta'}2}t^{\beta-2\ell-s}\norm{u}.\nonumber\end{equation} The first term on the right hand side is increasing in $t$, while the second is decreasing, so we can optimize by choosing $t=t(\lambda)$ making the two terms equal, that is $t=\lambda^{\frac{1}{2\Delta}},$ to obtain the claimed rate.
\end{proof}

%%%%%%%%% Section Posterior Consistency %%%%%%%%%%%

\section{Posterior Contraction}\label{sec:main}
In this section we employ the developments of the preceding sections to study the posterior consistency of the Bayesian solution to the inverse problem. That is, we consider a family of data sets $y^\dagger=y^\dagger(n)$ given by (\ref{eq:int10}) and study the limiting behavior of the posterior measure $\mu^{y^\dagger}_{\lambda,n}=\G(m_\lambda^\dagger,\C)$ as $n\to\infty$. Intuitively we would hope to recover a measure which concentrates near the true solution $u^\dagger$ in this limit. Following the approach in
~\cite{1232.62079},~\cite{MR1790007},~\cite{MR2418663} and~\cite{simoni2}, 
 we quantify this idea as in (\ref{eq:main1}). %as follows: we aim to determine $\varepsilon_n$ such that \begin{equation}\label{eq:main1}\E^{y^\dagger}\mu^{y^\dagger}_{\lambda,n}\left\{u:\norm{u-u^\dagger}\geq M_n\varepsilon_n\right\}\to0, \quad \forall M_n\to\infty,\end{equation}
%where the expectation is with respect to the random variable $y^\dagger$ distributed according to the data likelihood
% $\G(\A^{-1}u^\dagger, \frac1n\C_1)$.
By the Markov inequality we have \begin{equation}\E^{y^\dagger}\mu^{y^\dagger}_{\lambda,n}\left\{u:\norm{u-u^\dagger}\geq M_n\varepsilon_n\right\}\leq\frac1{M_n^2\varepsilon_n^2}\E^{y^\dagger}\int\norm{u-u^\dagger}^2\mu^{y^\dagger}_{\lambda,n}(du),\nonumber\end{equation}
so that it suffices to show that \begin{equation}\label{eq:main2}\E^{y^\dagger}\int\norm{u-u^\dagger}^2\mu^{y^\dagger}_{\lambda,n}(du)\leq c \varepsilon_n^2.\end{equation}
In addition to $n^{-1}$, there is a second small parameter in the problem, namely the regularization parameter, $\lambda=\frac{1}{n\tau^2}$, and we will choose a relationship between $n$ and $\lambda$ in order to optimize the convergence rates $\varepsilon_n$. We will show that determination of optimal convergence rates follows directly from the operator norm bounds on $\B_\lambda^{-1}$ derived in the previous section, which concern only $\lambda$ dependence; relating $n$ to $\lambda$ then follows as a trivial optimization. Thus, the $\lambda$ dependence of the operator norm bounds in the previous section forms the heart of the posterior contraction analysis. The relationship between $\lambda$ and $n$ will induce a relationship between $\tau$ and $n$, where $\tau$ being the scaling parameter in the prior covariance is the relevant parameter in the current Bayesian framework.

We now present our convergence results. In Theorem \ref{pdeth1} we study the convergence of the posterior mean to the true solution in a range of norms, while in Theorem \ref{pdeth2}  we study the concentration of the posterior near the true solution as described in (\ref{eq:main1}). The proofs of Theorems \ref{pdeth1} and \ref{pdeth2} are provided later in the current section. The two main convergence results, Theorems \ref{pdecor1} and \ref{pdecor2} follow as direct corollaries of Remark \ref{r1} and Theorems \ref{pdeth1} and \ref{pdeth2} respectively.

\begin{thm}\label{pdeth1} Let $u^\dagger\in X^1$.
Under the Assumptions \ref{a2}, we have that, for the choice $\tau=\tau(n)=n^\frac{\theta_2-\theta_1-1}{2(\theta_1-\theta_2+2)}$ and for any $\theta\in[0,1]$   \begin{equation}\E^{y^\dagger}\norm{m_\lambda^\dagger-u^\dagger}_{\eta}^2\leq cn^\frac{\theta+\theta_2-2}{\theta_1-\theta_2+2}, \nonumber\end{equation}
where $\eta=(1-\theta)(\beta-2\ell)+\theta$. The result holds for any $\theta_1, \theta_2\in[0,1]$, chosen so that $\,\E(\kappa^2)<\infty$, for $\kappa=\max\left\{\norm{\xi}_{{2\beta-2\ell-\eta_1}}, \norm{u^\dagger}_{{2-\eta_2}}\right\}$, where $\eta_i=(1-\theta_i)(\beta-2\ell)+\theta_i, \;i=1,2.$
\end{thm}

\begin{thm}\label{pdeth2}Let $u^\dagger\in X^1$.
Under the Assumptions \ref{a2}, we have that, for  $\tau=\tau(n)=n^\frac{\theta_2-\theta_1-1}{2(\theta_1-\theta_2+2)}$, the convergence in (\ref{eq:main1}) holds with \begin{equation}\varepsilon_n= n^{\frac{\theta_0+\theta_2-2}{2(\theta_1-\theta_2+2)}},\quad\theta_0=\left\{\begin{array}{ll} \frac{2\ell-\beta}{\Delta}, & if 
\;\mbox{$\beta-2\ell\leq0$} 
                                     \\ 0, &otherwise.
                                    \end{array}\right.
                                     \nonumber\end{equation} The result holds for any $\theta_1, \theta_2\in[0,1]$, chosen so that $\E(\kappa^2)<\infty$, for\\  $\kappa=\max\left\{\norm{\xi}_{{2\beta-2\ell-\eta_1}},  \norm{u^\dagger}_{{2-\eta_2}}\right\}$, where $\eta_i=(1-\theta_i)(\beta-2\ell)+\theta_i, \;i=1,2.$

\end{thm}

\begin{rem}\label{r1}\begin{enumerate}
\item[i)]To get convergence in the PDE method we need $\E\norm{u^\dagger}_{{2-\eta_2}}^2<\infty$ for a $\theta_2\leq1$. Under the a priori information that $u^\dagger\in X^\gamma$, we need $\gamma\geq2-\eta_2=1+(1-\theta_2)\Delta$ for some $\theta_2\in[0,1]$. Thus the minimum requirement for convergence is $\gamma=1$ in agreement to our assumption $u^\dagger\in X^1$. On the other hand, to obtain the optimal rate (which corresponds to choosing $\theta_2$ 
as small as possible) we need to choose $\theta_2=\frac{\Delta+1-\gamma}{\Delta}$. If $\gamma>1+\Delta$ then the right hand side is negative so we have to choose $\theta_2=0$, hence we cannot achieve the optimal rate. We say that the method saturates at $\gamma=1+\Delta$ which reflects the fact that the true solution has more regularity than the method allows us to exploit to obtain faster convergence rates.

\item[ii)]To get convergence we also need $\E\norm{\xi}_{{2\beta-2\ell-\eta_1}}^2<\infty$ for a $\theta_1\leq1$. By Lemma \ref{l2}(iii), it suffices to have  $\theta_1>\frac{s_0}\Delta$. This means that we need $\Delta>s_0$, which holds by the Assumption \ref{a2}(\ref{a2delta}), in order to be able to choose $\theta_1\leq1$. On the other hand, since $\Delta>0$ and $s_0\geq 0$, we have that $\frac{s_0}{\Delta}\geq0$ thus we can always choose $\theta_1$ in an optimal way, that is, we can always choose $\theta_1=\frac{s_0+\varepsilon}{\Delta}$ where $\varepsilon>0$ is arbitrarily small. %Note that for the convergence in (\ref{eq:main1}), the effect of $\varepsilon$ is absorbed in the sequence $M_n\to\infty$.
\item[iii)]If we want draws from $\mu_0$ to be in $X^\gamma$ then by Lemma \ref{l1}(ii) we need $1-s_0>\gamma$. Since the requirement for the method to give convergence is $\gamma\geq 1$ while $1-s_0\leq1$, we can never have draws exactly matching the regularity of the prior. On the other hand if we want an undersmoothing prior (which according to~\cite{1232.62079} in the diagonal case gives asymptotic coverage equal to 1) we need $1-s_0\leq\gamma$, which we always have. This, as discussed in Section \ref{sec:intro}, gives an explanation to the observation that in both of the above theorems we always have $\tau\to0$ as $n\to\infty$.
\item[iv)] When $\beta-2\ell>0,$ in Theorem \ref{pdeth2} and in Theorem \ref{pdecor2} below, we get suboptimal rates. The reason is that our analysis to obtain the error in the $\X$-norm is based on interpolating between the error in the $X^{\beta-2\ell}$-norm and the error in the $X^1$-norm. When $\beta-2\ell>0,$ interpolation is not possible since the $\X$-norm is now weaker than the $X^{\beta-2\ell}$-norm. However, we can at least bound the error in the $\X$-norm by the error in the $X^{\beta-2\ell}$-norm, thus obtaining a suboptimal rate. Note, that the case $\beta-2\ell>0$ does not necessarily correspond to the well posed case: by Lemma \ref{l2} we can only guarantee that a draw from the noise distribution lives in  $X^\rho, \;\rho<\beta-s_0$, while the range of $\A^{-1}$ is formally $X^{2\ell}$. Hence, in order to have a well posed problem we need $\beta-s_0>2\ell$, or equivalently $\Delta<1-s_0$. This can happen despite our assumption $\Delta>2s_0$, when $s_0<1/3$ and for appropriate choice of $\ell$ and $\beta$. In this case, regularization is unnecessary.
\end{enumerate}\end{rem}

Note that, since the posterior is Gaussian, the left hand side in (\ref{eq:main2}) is the Square Posterior Contraction %comment that it involves the approximating properties of the mean but also the uncertainty see bartek
\begin{equation}\label{eq:main3}SPC=\E^{y^\dagger}\norm{m_\lambda^\dagger-u^\dagger}^2+\tr(\C_{\lambda,n}),\end{equation}
which is the sum of the mean integrated squared error (MISE) and the posterior spread. Let $u^\dagger\in X^1$. By Lemma \ref{id}, the relationship (\ref{eq:int10}) between $u^\dagger$ and $y^\dagger$ and the equation (\ref{eq:int11}) for $m_\lambda^\dagger$, we obtain 
\begin{align*}\B_\lambda m_\lambda^\dagger=\A^{-1}\C_1^{-1}y^\dagger&=\A^{-1}\C^{-1}\A^{-1}u^\dagger+\frac{1}{\sqrt{n}}\A^{-1}\C^{-1}\xi\\\text{and}\quad
\B_\lambda u^\dagger&=\A^{-1}\C^{-1}\A^{-1}u^\dagger+\lambda\C_0^{-1}u^\dagger,\end{align*} where the equations hold in $X^{-1}$, since by a similar argument to the proof of Proposition \ref{meanlem} we have $m_\lambda^\dagger\in X^1$. By subtraction we get \begin{equation}\B_\lambda(m_\lambda^\dagger-u^\dagger)=\frac1{\sqrt{n}}\A^{-1}\C_1^{-1}\xi-\lambda \C_0^{-1}u^\dagger.\nonumber\end{equation} 

Therefore  \begin{equation}\label{eq:main5}m_\lambda^\dagger-u^\dagger=\B_\lambda^{-1}\left(\frac1{\sqrt{n}}\A^{-1}\C_1^{-1}\xi-\lambda \C_0^{-1}u^\dagger\right),\end{equation} as an equation in $X^1$. Using the fact that the noise has mean zero and the relation (\ref{eq:int6}), equation (\ref{eq:main5}) implies that we can split the square posterior contraction into three terms\begin{equation}\label{eq:main6}SPC=\norm{\lambda \B_\lambda^{-1}\C_0^{-1}u^\dagger}^2+\E\norm{\frac1{\sqrt{n}}\B_\lambda^{-1}\A^{-1}\C_1^{-1}\xi}^2+\frac1n\tr(B_{\lambda}^{-1}),\end{equation}
provided the right hand side is finite. A consequence of the proof of Theorem \ref{justth} is that $\B_\lambda^{-1}$ is trace class. 
Note that for $\zeta$ a white noise, we have that \begin{equation}\tr(\B_\lambda^{-1})=\E\norm{\B_\lambda^{-\frac12}\zeta}^2=\E\pr{\zeta}{\B_\lambda^{-1}\zeta}=\E\pr{\C_0^{\frac{s}2}\zeta}{\C_0^{-\frac{s}2}\B_\lambda^{-1}\C_0^{-\frac{s}2}\C_0^{\frac{s}2}\zeta}\nonumber\end{equation}\begin{equation}\leq \norm{\C_0^{-\frac{s}2}\B_\lambda^{-1}\C_0^{-\frac{s}2}}_{\mathcal{L}(\X)}\E\norm{\C_0^\frac{s}2\zeta}^2,\nonumber\end{equation} which for $s>s_0$ since by Lemma \ref{l1} we have that $\E\norm{\C_0^{\frac{s}2}\zeta}^2<\infty$, provides the bound \begin{equation}\label{eq:main7}\tr(\B_\lambda^{-1})\leq c\norm{\C_0^{-\frac{s}2}\B_\lambda^{-1}\C_0^{-\frac{s}2}}_{\mathcal{L}(\X)},\end{equation} where $c>0$ is independent of $\lambda$. 
If $q,r$ are chosen sufficiently large so that $\norm{\C_0^{-\frac{q}2-1}u^\dagger}<\infty$ and $\E\norm{\C_0^\frac{r}2\A^{-1}\C_1^{-1}\xi}^2<\infty$ then we see that \begin{equation}\label{eq:main8}SPC\leq c\left(\lambda^2\norm{ \B_\lambda^{-1}}_{\mathcal{L}(X^q,\X)}^2+\frac1n\norm{\B_\lambda^{-1}}_{\mathcal{L}(X^{-r},\X)}^2
+\frac1n\norm{\C_0^{-\frac{s}2}\B_\lambda^{-1}\C_0^{-\frac{s}2}}_{\mathcal{L}(\X)}\right), \end{equation}  where $c>0$ is independent of $\lambda$ and $n$. Thus identifying $\varepsilon_n$ in (\ref{eq:main1}) can be achieved simply through properties of the inverse of $\B_\lambda$ and its parametric dependence on $\lambda$.  

In the following, we are going to study convergence rates for the square posterior contraction, (\ref{eq:main6}), which by the previous analysis will secure that \begin{equation}\E^{y^\dagger}\mu^{y^\dagger}_{\lambda,n}\left\{u\colon \norm{u-u^\dagger}\geq \varepsilon_n\right\}\to0,\nonumber\end{equation} for $\varepsilon_n^2\to0$ at a rate almost as fast as the square posterior contraction. This suggests that the error is determined by the MISE and the trace of the posterior covariance, thus we optimize our analysis with respect to these two quantities. In~\cite{1232.62079} the situation where $\C_0, \C_1$ and $\A$ are diagonalizable in the same eigenbasis is studied, and it is shown that the third term in equation (\ref{eq:main6}) is bounded by the second term in terms of their parametric dependence on $\lambda$. The same idea is used in the proof of Theorem \ref{pdeth2}.

We now provide the proofs of Theorem \ref{pdeth1} and Theorem \ref{pdeth2}.

\begin{proof}[Proof of Theorem \ref{pdeth1}]
Since $\xi$ has zero mean, we have by (\ref{eq:main5}) \begin{equation}\E\norm{m_\lambda^\dagger-u^\dagger}_{{\beta-2\ell}}^2=\lambda^2\norm{\B_\lambda^{-1}\C_0^{-1}u^\dagger}_{{\beta-2\ell}}^2+
\frac1n\E\norm{\B_\lambda^{-1}\A^{-1}\C_1^{-1}\xi}_{{\beta-2\ell}}^2\nonumber\end{equation}and \begin{equation}\E\norm{m_\lambda^\dagger-u^\dagger}_1^2=\lambda^2\norm{\B_\lambda^{-1}\C_0^{-1}u^\dagger}_{1}^2+
\frac1n\E\norm{\B_\lambda^{-1}\A^{-1}\C_1^{-1}\xi}_{1}^2.\nonumber\end{equation}
Using Proposition \ref{pdelem1} and Assumption \ref{a2}(\ref{a2v}), we get \begin{equation}\E\norm{m_\lambda^\dagger-u^\dagger}_{{\beta-2\ell}}^2\leq c\E(\kappa^2)(\lambda^{2-\theta_2}+\frac1n\lambda^{-\theta_1})=c\E(\kappa^2)(n^{\theta_2-2}\tau^{2\theta_2-4}+n^{\theta_1-1}\tau^{2\theta_1})\nonumber\end{equation}and
  \begin{equation}\E\norm{m_\lambda^\dagger-u^\dagger}_{1}^2\leq c\E(\kappa^2)(\lambda^{1-\theta_2}+\frac1n\lambda^{-\theta_1-1})=\frac{c\E(\kappa^2)}{\lambda}(n^{\theta_2-2}\tau^{2\theta_2-4}+n^{\theta_1-1}\tau^{2\theta_1}).\nonumber\end{equation}
Since the common parenthesis term, consists of a decreasing and an increasing term in $\tau$, we optimize the rate by choosing $\tau=\tau(n)=n^p$ such that the two terms become equal, that is, $p=\frac{\theta_2-\theta_1-1}{2(\theta_1-\theta_2+2)}$. 
We obtain, \begin{equation}\E\norm{m_\lambda^\dagger-u^\dagger}_{{\beta-2\ell}}^2\leq c\E(\kappa^2)n^\frac{\theta_2-2}{\theta_1-\theta_2+2}\quad\text{and}\quad\E\norm{m_\lambda^\dagger-u^\dagger}_{1}^2\leq c\E(\kappa^2)n^\frac{\theta_2-1}{\theta_1-\theta_2+2}.\nonumber\end{equation} By interpolating between the two last estimates we obtain the claimed rate.
\end{proof}

\begin{proof}[Proof of Theorem \ref{pdeth2}] Recall equation (\ref{eq:main6})\begin{equation}SPC=\norm{\lambda \B_\lambda^{-1}\C_0^{-1}u^\dagger}^2+\E\norm{\frac1{\sqrt{n}}\B_\lambda^{-1}\A^{-1}\C_1^{-1}\xi}^2+\frac1n\tr(B_{\lambda}^{-1}).\nonumber\end{equation} 
The idea is that the third term is always dominated by the second term. 
 Combining equation (\ref{eq:main7}) with  Proposition \ref{pdelem3}, we have that \begin{equation}\frac1n \tr(\B_\lambda^{-1})\leq c\frac1n\lambda^{-\frac{2\ell-\beta+s}{\Delta}}, \;\forall s\in(\{\beta-2\ell\}\vee \{s_0\},1].\nonumber\end{equation}
\begin{enumerate}
\item[i)] Suppose $\beta-2\ell\leq0,$ so that by Proposition \ref{pdelem1} we have, where $\theta_0=\frac{2\ell-\beta}{\Delta}\in[0,1]$, using Assumption \ref{a2}(\ref{a2v})
\begin{equation}\E\norm{\frac{1}{\sqrt{n}}\B_\lambda^{-1}\A^{-1}\C_1^{-1}\xi}^2\leq c\frac{1}n\E\norm{\xi}^2_{{2\beta-2\ell-\eta_1}}\lambda^{-\theta_1-\theta_0}\nonumber\end{equation}and \begin{equation}\norm{\lambda\B_\lambda^{-1}\C_0^{-1}u^\dagger}^2\leq c\norm{u^\dagger}^2_{2-\eta_2}\lambda^{2-\theta_2-\theta_0}.\nonumber\end{equation} Note that $\theta_1$ is chosen so that $\E\norm{\xi}_{{2\beta-2\ell-\eta_1}}^2<\infty$, that is, by Lemma \ref{l2}(iii), it suffices to have $\theta_1>\frac{s_0}{\Delta}$. Noticing that by choosing $s$ arbitrarily close to $s_0,$ we can have $\frac{2\ell-\beta+s}{\Delta}$  arbitrarily close to $\frac{2\ell-\beta+s_0}{\Delta}$, and since $\theta_1+\theta_0>\frac{2\ell-\beta+s_0}{\Delta}$, we deduce that the third term in equation (\ref{eq:main6}) is always dominated by the second term. Combining, we have that \begin{equation}SPC\leq \frac{c\E(\kappa^2)}{\lambda^{\theta_0}} (\lambda^{2-\theta_2}+\frac{1}n\lambda^{-\theta_1})=\frac{c\E(\kappa^2)}{\lambda^{\theta_0}}(n^{\theta_2-2}\tau^{2\theta_2-4}+n^{\theta_1-1}\tau^{2\theta_1}).\nonumber\end{equation} 
 \item[ii)] Suppose $\beta-2\ell>0$. Using Proposition \ref{pdelem1} and Assumption \ref{a2}(\ref{a2v})  we have
 \begin{equation}\norm{\lambda \B_\lambda^{-1}\C_0^{-1}u^\dagger}^2\leq c\norm{\lambda \B_\lambda^{-1}\C_0^{-1}u^\dagger}^2_{{\beta-2\ell}}\leq c\norm{u^\dagger}^2_{{2-\eta_2}}\lambda^{2-\theta_2}\nonumber\end{equation}
 and \begin{equation}\E\norm{\frac{1}{\sqrt{n}}\B_\lambda^{-1}\A^{-1}\C_1^{-1}\xi}^2\leq c\E\norm{\frac{1}{\sqrt{n}}\B_\lambda^{-1}\A^{-1}\C_1^{-1}\xi}_{{\beta-2\ell}}^2\nonumber\end{equation}\begin{equation}\leq c\frac{1}n\E\norm{\xi}^2_{{2\beta-2\ell-\eta_1}}\lambda^{-\theta_1},\nonumber\end{equation}where as before $\theta_1>\frac{s_0}{\Delta}.$ 
 The third term in equation (\ref{eq:main6}) is again dominated by the second term, since on the one hand $\theta_1>\frac{s_0}{\Delta}$ and on the other hand, since $\beta-2\ell>0$, we can always choose $\{\beta-2\ell\}\vee\{s_0\}<s\leq 1\wedge \{s_0+\beta-2\ell\}$ to get $\frac{2\ell-\beta+s}{\Delta}\leq\frac{s_0}{\Delta}.$ 
Combining the three estimates we have that \begin{equation}SPC\leq c\E(\kappa^2)(n^{\theta_2-2}\tau^{2\theta_2-4}+n^{\theta_1-1}\tau^{2\theta_1}).\nonumber\end{equation} \end{enumerate}
In both cases, the common term in the parenthesis consists of a decreasing and an increasing term in $\tau$, thus we can optimize by choosing $\tau=\tau(n)=n^p$ making the two terms equal, that is, $p=\frac{\theta_1-\theta_2+1}{2\theta_2-2\theta_1-4}$, to get the claimed rates. 
\end{proof}

%%%%%%%%%%%%% Section Example %%%%%%%%%%%%%%%

\section{Examples}\label{sec:ex}
We now present some nontrivial examples satisfying Assumptions \ref{a2}.

Let $\Omega\subset \R^d, \;d=1,2,3,$ be a bounded and open set. We define $\A_0\coloneq -\boldsymbol\Delta$, where $\boldsymbol\Delta$ is the Dirichlet Laplacian which is the Friedrichs extension of the classical Laplacian defined on $C^2_0(\Omega),$ that is, $\A_0$ is a self-adjoint operator with a domain $\D(\A_0)$ dense in $\X\coloneq L^2(\Omega)$~\cite{MR1892228}. For $\partial\Omega$  sufficiently smooth we have $\D(\A_0)=H^2(\Omega)\cap H_0^1(\Omega)$. It is well known that $\A_0$ has a compact inverse and that it possesses an eigensystem $\{\rho_k^2,e_k\}_{k=1}^\infty,$ where the eigenfunctions $\{e_k\}$ form a complete orthonormal basis of $\X$ and the eigenvalues $\rho_k^2$ behave asymptotically like $k^{\frac2d}$~\cite{MR2192832}.

In Subsections \ref{ex1} and \ref{ex2}, we consider the inverse problem to find $u$ from $y$, where \[y=z+\frac{1}{\sqrt{n}}\xi,\]for $z$ solving the partial differential equation \begin{align*}-\boldsymbol\Delta z+qz&=u \quad\mbox{in}\quad\Omega,\\
z&=0\quad\mbox{on}\quad\partial\Omega,\end{align*} that is, $\A_0z+qz=u$, where $q$ is a nonnegative real function of certain regularity. We choose prior and noise distributions with covariance operators which are not simultaneously diagonalizable with the forward operator. Later on, in Subsection \ref{ex3}, we consider more complicated examples and in particular, we consider fractional powers of the Dirichlet Laplacian in the forward operator, as well as more general choices of prior and noise covariance operators. 

Our general strategy for proving the validity of our norm equivalence assumptions is:
\begin{enumerate}\item[i)]if needed, use Proposition \ref{inter} below to reduce the range of spaces required to check an assumption's validity to a finite set of spaces;
\item[ii)]reformulate the assumptions as statements regarding the boundedness of operators of the form considered in Lemma \ref{prolem} below.
\end{enumerate}
The statement of Proposition \ref{inter}, which is a well known result from interpolation theory, and the statement and proof of Lemma \ref{prolem} are postponed to Subsection \ref{interpolation}.

\subsection{Example 1 - Non-diagonal forward operator}\label{ex1}
We study the Bayesian inversion of the operator $\A^{-1}\coloneq (\A_0+\M)^{-1}$ where $\M\colon L^2(\Omega)\to L^2(\Omega)$ is the multiplication operator by a nonnegative function $q\in W^{2,\infty}(\Omega)$. We assume that the observational noise is white, so that $\C_1=I$, and we set the prior covariance operator to be $\C_0=\A_0^{-2}$.

The operator $\C_0$ is trace class. Indeed, let $\lambda_k^2=\rho_k^{-4}$ be its eigenvalues. Then they behave asymptotically like $k^{-\frac4d}$ and $\sumi k^{-\frac4d}<\infty$ for $d<4$. Furthermore, we have that $\sumi\lambda_k^{2s}\leq c\sumi k^{-\frac{4s}{d}}<\infty,$ provided $s>\frac{d}4$, that is, the Assumption \ref{a2}(\ref{a1}) is satisfied with \begin{equation}s_0=\left\{\begin{array}{ll} $1/4$, &
\;\mbox{$d=1,$} 
                                     \\ 1/2, & \;\mbox{$d=2,$}
                                     \\ 3/4, & \;\mbox{$d=3$.}
                                    \end{array}\right.
                                     \nonumber\end{equation}
We define the Hilbert scale induced by $\C_0=\A_0^{-2}$, that is, $(X^s)_{s\in\R}$, for $X^s\coloneq \overline{\mathcal{M}}^{\|\cdot\|_{s}}$, where\begin{equation}\mathcal{M}=\bigcap_{l=0}^\infty\D(\A_0^{2l}),\;\pr{u}{v}_{s}\coloneq \pr{\A_0^{s}u}{\A_0^{s}v} \quad\text{and}\quad {\|u\|}_{s}\coloneq \norm{\A_0^{s}u}.\nonumber\end{equation} Observe, $X^0=\X=L^2(\Om)$.                                                           
                                     
Our aim is to show that $\C_1\simeq \C_0^{\beta}$ and $\A^{-1}\simeq \C_0^{\ell}$, where $\beta=0$ and $\ell=\frac12$, in the sense of the Assumptions \ref{a2}. We have $\Delta=2\ell-\beta+1=2$. Since for $d=1,2,3$ we have $0<s_0<1$, the Assumption \ref{a2}(\ref{a2delta}) is satisfied. Moreover, note that since $\C_1=I$ the Assumptions \ref{a2}(\ref{a2ii}) and (\ref{a2iii}) are trivially satisfied.                               

We now show that Assumptions \ref{a2} (\ref{a2i}), (\ref{a2iv}), (\ref{a2v}) are also satisfied. In this example the three assumptions have the form 
\begin{enumerate}
\item[3.]$\norm{(\A_0+\M)^{-1}u}\asymp\norm{\A_0^{-1}u}, \;\forall u\in X^{-1};$
\item[6.]$\norm{\A_0^s(\A_0+\M)^{-1}u}\leq c_3\norm{\A_0^{s-1}u}, \;\forall u\in X^{s-1}, \;\forall s\in(s_0,1]$;
\item[7.]$\norm{\A_0^{-\eta}(\A_0+\M)^{-1}u}\leq c_4\norm{\A_0^{-\eta-1}u}, \;\forall u\in X^{-\eta-1}, \;\forall \eta\in[-1,1]$. 
\end{enumerate}
Observe that Assumption (\ref{a2iv}) is implied by Assumption (\ref{a2v}). 
\begin{prop}\label{proass}The Assumptions \ref{a2} are satisfied in this example.
\end{prop}
\begin{proof}
We only need to show that Assumptions (\ref{a2i}) and (\ref{a2v}) hold.
\begin{enumerate}
\item[3.]The assumption is equivalent to $\T\coloneq(\A_0+\M)^{-1}\A_0$ and $\T^{-1}=\A_0^{-1}(\A_0+\M)$ being bounded in $\X$. Since  $\T^{-1}=I+\A_0^{-1}\M$ which is bounded in $\X$, we only need to show that $\T$ is bounded. Indeed, $(\A_0+\M)^{-1}\A_0=(I+\A_0^{-1}\M)^{-1}$, which is bounded by Lemma \ref{prolem} applied for $t=-1, s=1$. 
\item[7.] By Proposition \ref{inter}, it suffices to show $\T\in\mathcal{L}(\X)\cap\mathcal{L}(X^1)\cap\mathcal{L}(X^{-1}).$ We have already shown that $\T\in\mathcal{L}(\X)$. For $\T\in\mathcal{L}(X^1),$ note that it is equivalent to $\A_0\T\A_0^{-1}=(I+\M\A_0^{-1})^{-1}\in\mathcal{L}(\X)$, which holds by Lemma \ref{prolem} applied for $t=s=1$. Finally, for $\T\in\mathcal{L}(X^{-1})$, note that it is equivalent to $\A_0^{-1}\T\A_0=(I+\A_0^{-2}\M\A_0)^{-1}\in\mathcal{L}(\X)$, which holds by Lemma \ref{prolem} applied for $t=-1, s=1$.
\end{enumerate}
\end{proof}

We can now apply Theorem \ref{pdecor1} and Theorem \ref{pdecor2} to get the following convergence result.

\begin{thm}
Let $u^\dagger\in X^\gamma, \gamma\geq1$. Then, for $\tau=\tau(n)=n^\frac{4-d-4(\gamma\wedge3)-\varepsilon}{8(\gamma\wedge3)+8+2d+2\varepsilon}$, the convergence in (\ref{eq:main1}) holds with $\varepsilon_n=n^{-e}$, where \begin{equation}e=\left\{\begin{array}{ll} \frac{2\gamma}{4+d+4\gamma+2\varepsilon}
, & if 
\;\mbox{$\gamma<3$} 
                                     \\ \frac{6}{16+d+2\varepsilon}, & if\;\mbox{$\gamma\geq3$,}\end{array}\right.\nonumber\end{equation}for $\varepsilon>0$ arbitrarily small and where $d=1,2,3,$ is the dimension. Furthermore, for $t\in[-1,1)$, for the same choice of $\tau$, we have $\E\norm{m_\lambda^\dagger-u^\dagger}^2_t\leq cn^{-h},$ where \begin{equation}h=\left\{\begin{array}{ll} \frac{4\gamma-4t}{4+d+4\gamma+2\varepsilon}
, & if 
\;\mbox{$\gamma<3$} 
                                     \\ \frac{12-4t}{16+d+2\varepsilon}, & if\;\mbox{$\gamma\geq3$.}\end{array}\right.\nonumber\end{equation} For $t=1$ the above rate holds provided $\gamma>1$.
\end{thm}

\subsection{Example 2 - A fully non-diagonal example}\label{ex2}
As in Example \ref{ex1}, we study the Bayesian inversion of the operator $\A^{-1}=(\A_0+\M)^{-1}$ for a nonnegative $q\in W^{2,\infty}(\Omega)$. We assume that the observational noise is Gaussian with covariance operator $\C_1\coloneq(\A_0^{\frac14}+\Mr)^{-2}$, where $\Mr: L^2(\Omega)\to L^2(\Omega)$ is the multiplication operator by a nonnegative function $r\in W^{4,\infty}(\Omega)$. As before, we set the prior covariance operator to be $\C_0=\A_0^{-2}$, thus the Assumption \ref{a2}(\ref{a1}) is satisfied with the same $s_0$ and we work in the same Hilbert scale $(X^s)_{s\in\R}.$ 

We show that $\C_1\simeq \C_0^\beta$ and $\A^{-1}\simeq \C_0^\ell$, where $\beta=\frac14$ and $\ell=\frac12$, in the sense of the Assumptions \ref{a2}(\ref{a2i})-(\ref{a2v}). First note that we have $\Delta=2\ell-\beta+1=\frac74>2s_0$ for $d=1,2,3$, so that the Assumption \ref{a2}(\ref{a2delta}) is satisfied. The rest of the assumptions have the form 
\begin{enumerate}
\item[3.]$\norm{(\A_0^\frac14+\Mr)(\A_0+\M)^{-1}u}\asymp\norm{\A_0^{-\frac34}u}, \;\forall u\in X^{-\frac34};$
\item[4.]$\norm{\A_0^{\rho}(\A_0^\frac14+\Mr)^{-1}u}\leq c_1\norm{\A_0^{\rho-\frac14}u}, \;\forall u\in X^{\rho-\frac14}, \forall \rho\in[\lceil-s_0-\frac34\rceil,\frac14-s_0);$
\item[5.]$\norm{\A_0^{-s}(\A_0^\frac14+\Mr)u}\leq c_2\norm{\A_0^{\frac14-s}}, \;\forall u\in X^{\frac14-s}, \forall s\in(s_0,1];$
\item[6.]$\norm{\A_0^s(\A_0^\frac14+\Mr)(\A_0+\M)^{-1}u}\leq c_3\norm{\A_0^{s-\frac34}u}, \;\forall u\in X^{s-\frac34}, \;\forall s\in(s_0,1]$;
\item[7.]$\norm{\A_0^{-\eta}(\A_0+\M)^{-1}(\A_0^\frac14+\Mr)^2u}\leq c_4\norm{\A_0^{-\eta-\frac12}u}, \;\forall u\in X^{-\eta-\frac12}, \;\forall \eta\in[-\frac34,1]$. 
\end{enumerate}

\begin{prop} The Assumptions \ref{a2} are satisfied in this example.
\end{prop}
\begin{proof}We have already seen that the first two assumptions are satisfied.
\begin{enumerate}
\item[3.] We need to show that $\Ss\coloneq\Sr\Sq^{-1}\A_0^\frac34$ and $\Ss^{-1}$ are bounded operators in $\X$. Indeed, $\Ss=(I+\Mr\A_0^{-\frac14})(I+\A_0^{-\frac34}\M\A_0^{-\frac14})^{-1}$ which is bounded by Lemma \ref{prolem} applied for $t=s=\frac14$ and $t=\frac14, s=1$. For $\Ss^{-1}$ we have, $\Ss^{-1}=(I+\A_0^{-\frac34}\M\A_0^{-\frac14})(I+\Mr\A_0^{-\frac14})^{-1}$, which again by Lemma \ref{prolem} is the composition of two bounded operators.
 \item[4.] Since $\frac14-s_0=0, -\frac14, -\frac12$ for $d=1,2,3$ respectively, it suffices to show that it holds for all $\rho\in[-1,0]$. By Proposition \ref{inter} it suffices to show that $\mathcal{S}\coloneq\Sr^{-1}\A_0^\frac14\in\mathcal{L}(\X)\cap\mathcal{L}(X^{-1})$. This is equivalent to showing that $\Ss=(I+\A_0^{-\frac14}\Mr)^{-1}$ and $\A_0^{-1}\Ss\A_0=(I+\A_0^{-\frac54}\Mr\A_0)^{-1}$ are bounded in $\X$, which holds by Lemma \ref{prolem}.
\item[5.] By Proposition \ref{inter} it suffices to show that $\Ss\coloneq\Sr\A_0^{-\frac14}\in\mathcal{L}(\X)\cap\mathcal{L}(X^{-1})$. Indeed, $\Ss=I+\Mr\A_0^{-\frac14}\in\mathcal{L}(\X)$. On the other hand, to show $\Ss\in\mathcal{L}(X^{-1})$ it is equivalent to show that $\A_0^{-1}\Ss\A_0\in\mathcal{L}(X)$. Indeed,  $\A_0^{-1}\Ss\A_0=I+\A_0^{-1}\Mr\A_0^\frac34$ which is bounded by Lemma \ref{prolem}.

\item[6.] By Proposition \ref{inter} it suffices to show that $\Ss\coloneq\Sr\Sq^{-1}\A_0^\frac34\in\mathcal{L}(\X)\cap\mathcal{L}(X^1)$. Indeed, we have already shown in part $(3)$ of the current proof that $\Ss\in\mathcal{L}(\X)$. To show $\Ss\in\mathcal{L}(X^1)$ it is equivalent to show that $\A_0\Ss\A_0^{-1}\in\mathcal{L}(\X)$. Indeed, $\A_0\Ss\A_0^{-1}=(I+\A_0\Mr\A_0^{-\frac54})(I+\A_0^{\frac14}\M\A_0^{-\frac54})^{-1}$ which by Lemma \ref{prolem} is the composition of two bounded operators in $\X$..

\item[7.] By Proposition \ref{inter} it suffices to show that $\Ss\coloneq(\A_0+\M)^{-1}\Sr^2\A_0^{\frac12}\in\mathcal{L}(\X)\cap\mathcal{L}(X^{-1})\cap\mathcal{L}(X^1)$. We start by showing $\Ss\in\mathcal{L}(\X)$. Indeed, we have $\Ss=(I+\A_0^{-1}\M)^{-1}(I+\A_0^{-1}\Mr\A_0^{\frac34})(I+\A_0^{-\frac34}\Mr\A_0^\frac12)$, which by Lemma \ref{prolem}, is the composition of three bounded operators. For showing $\Ss\in\mathcal{L}(X^{-1})$ it is equivalent to show that $\A_0^{-1}\Ss\A_0\in\mathcal{L}(\X)$. Indeed, $\A_0^{-1}\Ss\A_0=(I+\A_0^{-2}\M\A_0)^{-1}(I+\A_0^{-2}\Mr\A_0^\frac74)(I+\A_0^{-\frac74}\Mr\A_0^\frac32)$, which by Lemma \ref{prolem}, is the composition of three bounded operators. Finally, we show that $\Ss\in\mathcal{L}(X^1)$ or equivalently $\A_0 \Ss\A_0^{-1}\in\mathcal{L}(\X)$. Indeed we have 
 $\A_0\Ss\A_0^{-1}=(I+\M\A_0^{-1})^{-1}(I+\Mr\A_0^{-\frac14})(I+\A_0^{\frac14}\Mr\A_0^{-\frac12})$, which again by Lemma \ref{prolem}, is the composition of three bounded operators.
\end{enumerate}
\end{proof}

We can now apply Theorem \ref{pdecor1} and Theorem \ref{pdecor2} to get the following convergence result.

\begin{thm}
Let $u^\dagger\in X^\gamma, \gamma\geq1$. Then, for $\tau=\tau(n)=n^\frac{4-d-(4\gamma\wedge 11)-\varepsilon}{(8\gamma\wedge22)+6+2d+2\varepsilon}$, the convergence in (\ref{eq:main1}) holds with $\varepsilon_n=n^{-e}$, where \begin{equation}e=\left\{\begin{array}{ll} \frac{2\gamma}{3+d+4\gamma+2\varepsilon}
, & if 
\;\mbox{$\gamma<\frac{11}4$} 
                                     \\ \frac{11}{28+2d+2\varepsilon}, & if\;\mbox{$\gamma\geq \frac{11}4$,}\end{array}\right.\nonumber\end{equation}for $\varepsilon>0$ arbitrarily small and where $d=1,2,3,$ is the dimension. Furthermore, for $t\in[-\frac34,1)$, for the same choice of $\tau$, we have $\E\norm{m_\lambda^\dagger-u^\dagger}^2_t\leq cn^{-h},$ where \begin{equation}h=\left\{\begin{array}{ll} \frac{4\gamma-4t}{3+d+4\gamma+2\varepsilon}
, & if 
\;\mbox{$\gamma<\frac{11}4$} 
                                     \\ \frac{22-8t}{28+2d+2\varepsilon}, & if\;\mbox{$\gamma\geq\frac{11}4$.}\end{array}\right.\nonumber\end{equation}
For $t=1$ the above rate holds provided $\gamma>1$.
\end{thm}

\subsection{Example 3 - More general lower order perturbations case}\label{ex3}
The same methodology can be applied to more general examples, for instance, in the case where $\A=\A_0^{\ell\alpha}+\M$, $\C_1=\Sb^{-2}$ and $\C_0=\A_0^{-\alpha}$, for nonnegative functions $q\in W^{a_q,\infty}(\Omega)$ and $r\in W^{a_r,\infty}(\Omega)$, where $\ell, \beta>0$ and  $\alpha>\frac{d}2$ such that $\Delta>2s_0=\frac{d}{\alpha}$. The values of $a_r, a_q$ are chosen as sufficiently large even integers depending on the values of $\alpha, \beta, \ell$. Note that we require $\ell, \beta>0$ for our compactness arguments to work, however, the cases $\beta=0$ and/or $\ell=0$ also work using a slightly modified proof. The proof is omitted for brevity but the
interested reader may consult~\cite{THESIS} for details.

\subsection{Technical results from interpolation theory}\label{interpolation}
Let $(Y^s)_{s\in\R}$ be the Hilbert scale induced by a self-adjoint positive definite linear operator $\mathcal{Q}\in\mathcal{L}(\X)$ (cf. Section \ref{sec:assumptions}). The following result holds~\cite[Theorems 4.36, 1.18, 1.6]{MR2523200}:

\begin{prop}\label{inter}
For any $t>0$, the couples $(\X,Y^t)$  and $(\X, Y^{-t})$ are interpolation couples and for every $\theta\in[0,1]$ we have  $(\X,Y^t)_{\theta,2}=Y^{\theta t}$ and  $(\X,Y^{-t})_{\theta,2}=Y^{-\theta t}$. In particular, for any $s\in\R$, if $T\in\mathcal{L}(\X)\cap\mathcal{L}(Y^s)$  then $T\in\mathcal{L}(Y^{\theta s})$ for any $\theta\in[0,1]$.
\end{prop}

Let $w\in W^{a_w,\infty}(\Omega)$ be a nonnegative function and define the multiplication operator $\mathcal{M}_w\colon \X\to\X$. Note that by the H\"older inequality the operator $\mathcal{M}_w$ is bounded. The last proposition, implies the following lemma.
\begin{lem}\label{prolem}
For any $t\in\R$, $\A_0^{t}\mathcal{M}_w\A_0^{-t}$ is a bounded operator in $\X$, provided $a_w\geq 2\lceil|t|\rceil$. Furthermore, for any $s>0$ the operators $K_1\coloneq\A_0^t\mathcal{M}_w\A_0^{-t-s}$ and $K_2\coloneq\A_0^{t-s}\mathcal{M}_w\A_0^{-t}$ are compact in $\X$ and $(I+K_i)^{-1}, i=1,2,$ are bounded in $\X$.
\end{lem}
\begin{proof}
We begin by showing that $\A_0^t\mathcal{M}_w\A_0^{-t}\in\mathcal{L}(\X),$ for  $t\in[-1,1]$. By the last proposition applied for $\mathcal{Q}=\A_0^{-2}, \;T=\mathcal{M}_w,$ and since $\mathcal{M}_w$ is bounded, it suffices to show that $\A_0^{-1}\mathcal{M}_w\A_0$
and $\A_0\mathcal{M}_w\A_0^{-1}$ are bounded in $\X$. In fact it suffices to show that $\A_0\mathcal{M}_w\A_0^{-1}$ is bounded since $\norm{\A_0^{-1}\mathcal{M}_w\A_0}=\norm{(\A_0^{-1}\mathcal{M}_w\A_0)^\ast}=\norm{\A_0\mathcal{M}_w\A_0^{-1}}.$
Indeed, since $\A_0=-\boldsymbol\Delta,$ \begin{equation}\norm{\A_0\mathcal{M}_w\A_0^{-1}\phi}=\norm{{\boldsymbol\Delta}\mathcal{M}_w\A_0^{-1}\phi}=\norm{({\boldsymbol\Delta} w)\A_0^{-1}\phi+2(\nabla w)\cdot(\nabla\A_0^{-1}\phi)+w{\boldsymbol\Delta}\A_0^{-1}\phi}\nonumber\end{equation}\begin{equation}\leq\norm{w}_{W^{2,\infty}(\Om)}(\norm{\A_0^{-1}\phi}+\norm{\nabla\A_0^{-1}\phi}+\norm{\phi})\leq c\norm{w}_{W^{2,\infty}(\Om)}\norm{\phi}.\nonumber\end{equation}
For general $t\in\R$, let $\kappa=\lceil|t|\rceil\in\N$, then as before it suffices to show that $\A_0^\kappa\mathcal{M}_w\A_0^{-\kappa}$ is bounded in $\X$. Again, using the fact that $\A_0=-\boldsymbol{\Delta}$, we have by the product rule for derivatives that $\A_0^\kappa\mathcal{M}_w\A_0^{-\kappa}$ is bounded, provided $w\in W^{2\kappa,\infty}(\Omega)$.

The operators $K_i$ are compact in $\X$, since they are compositions between the compact operator $\A_0^{-s}$ and the bounded operator $\A_0^{t}\mathcal{M}_w\A_0^{-t}$. Positivity of the operator $\A_0$ and nonnegativity of the operator $\mathcal{M}_w$ show that $-1$ cannot be an eigenvalue of $K_i$, so that by the Fredholm Alternative~\cite[\textsection27, Theorem 7]{MR0243367} we have that $(I+K_i)^{-1}, i=1,2,$ are bounded in $\X$.
\end{proof}

%%%%%%%%%%%%%%% Section The Diagonal Case %%%%%%%%%%%%

\section{The Diagonal Case}\label{sec:diag}In the case where $\C_0, \C_1$ and $\A$, are all diagonalizable in the same eigenbasis our assumptions are trivially satisfied, provided $\Delta>2s_0$. In~\cite{1232.62079}, sharp convergence rates are obtained for the convergence in (\ref{eq:main1}), in the case where the three relevant operators are simultaneously diagonalizable and have spectra that decay algebraically; the authors only consider the case $\C_1=I$ since in this diagonal setting the colored noise problem can be reduced to the white noise one. The rates in~\cite{1232.62079} agree with the minimax rates provided the scaling of the prior is optimally chosen,~\cite{MR2421941}. In Figure \ref{fig} (cf. Section \ref{sec:mainresults}) we have in green the rates of convergence predicted by Theorem \ref{pdecor2} and in blue the sharp convergence rates from~\cite{1232.62079}, plotted against the regularity of the true solution, $u^\dagger\in X^\gamma$, in the case where $\beta=\ell=\frac12$ and $\C_0$ has eigenvalues that decay like $k^{-2}$. In this case $s_0=\frac12$ and $\Delta=\frac32$, so that $\Delta>2s_0$. 

As explained in Remark \ref{r1}, the minimum regularity for our method to work is $\gamma=1$ and our rates saturate at $\gamma=1+\Delta$, that is, in this example at $\gamma=2.5$. We note that for $\gamma\in[1,2.5]$ our rates agree, up to $\varepsilon>0$ arbitrarily small, with the sharp rates obtained in~\cite{1232.62079}, for $\gamma>2.5$ our rates are suboptimal and for $\gamma<1$ the method fails. In~\cite{1232.62079}, the convergence rates are obtained for $\gamma>0$ and the saturation point is at $\gamma=2\Delta$, that is, in this example at $\gamma=3$. In general the PDE method can saturate earlier (if $2\ell-\beta>0$), at the same time (if $2\ell-\beta=0$), or later (if $2\ell-\beta<0$) compared to the diagonal method presented in~\cite{1232.62079}. However, the case $2\ell-\beta<0$ in which our method saturates later, is also the case in which our rates are suboptimal, as explained in Remark \ref{r1}(iv). 

The discrepancies can be explained by the fact that in Proposition \ref{pdelem1}, the choice of $\theta$ which determines both the minimum requirement on the regularity of $u^\dagger$ and the saturation point, is the same for both of the operator norm bounds. This means that on the one hand  to get convergence of the term $\norm{\lambda \B_\lambda^{-1}\C_0^{-1}u^\dagger}$ in equation (\ref{eq:main6}) in the proof of Theorem \ref{pdeth2}, we require conditions which secure the convergence in the stronger $X^1$-norm and on the other hand the saturation rate for this term is the same as the saturation rate in the weaker $X^{\beta-2\ell}$-norm. For example, when $\beta-2\ell=0$ the saturation rate in the PDE method is the rate of the $\X$-norm hence we have the same saturation point as the rates in~\cite{1232.62079}. In particular, we have agreement of the saturation rate when $\beta=\ell=0$, which corresponds to the problem  where we directly observe the unknown function polluted by white noise (termed the \emph{white noise model}).

\section{Conclusions}\label{sec:conclusions}
We have presented a new method of identifying the posterior distribution in a conjugate Gaussian Bayesian linear inverse problem setting (Section \ref{sec:mainresults} and Section \ref{sec:justification}). We used this identification to examine the posterior consistency of the Bayesian approach in a frequentist sense (Section \ref{sec:mainresults} and Section \ref{sec:main}). We provided convergence rates for the convergence of the expectation of the mean error in a range of norms (Theorem \ref{pdeth1}, Theorem \ref{pdecor1}). We also provided convergence rates for the square posterior contraction (Theorem \ref{pdeth2}, Theorem \ref{pdecor2}). Our methodology assumed a relation between the prior covariance, the noise covariance and the forward operator, expressed in the form of norm equivalence relations (Assumptions \ref{a2}). We considered Gaussian noise which can be white. In order for our methods to work we required a certain degree of ill-posedness compared to the regularity of the prior (Assumption \ref{a2}(\ref{a2delta})) and for the convergence rates to be valid a certain degree of regularity of the true solution. In the case where the three involved operators are all diagonalizable in the same eigenbasis, when the problem is sufficiently ill-posed with respect to the prior, and for a range of values of $\gamma$, the parameter expressing the regularity of the true solution, our rates agree (up to $\varepsilon>0$ arbitrarily small) with the sharp (minimax) convergence rates obtained in~\cite{1232.62079} (Section \ref{sec:diag}). 

Our optimized rates rely on rescaling the prior depending on the size of the noise, achieved by choosing the scaling parameter $\tau^2$ in the prior covariance as an appropriate function of the parameter $n^{-\frac12}$ multiplying the noise. However, the relationship between $\tau$ and $n$ depends on the unknown regularity of the true solution $\gamma$, which raises the question how to optimally choose $\tau$ in practice. An attempt to address this question in a similar but more restrictive setting than ours is taken in~\cite{simoni2}, where an empirical Bayes maximum likelihood based procedure giving a data driven selection of $\tau$ is presented. A different approach is taken in~\cite{knapik} in the simultaneously diagonalizable case. As discussed in~\cite{1232.62079}, for a fixed value of $\tau$ independent of $n$, the rates are optimal only if the regularity of the prior exactly matches the regularity of the truth. In~\cite{knapik}, an empirical Bayes maximum likelihood based procedure and a hierarchical method are presented providing data driven choices of the regularity of the prior, which are shown to give optimal rates up to slowly varying terms. We currently investigate hierarchical methods with conjugate priors and hyperpriors for data driven choices of both the scaling parameter of the prior $\tau$ and the noise level $n^{-\frac12}$.

The methodology presented in this paper is extended to drift estimation for diffusion processes in~\cite{Pokern:2012fk}. Future research includes the extension to an abstract setting which includes both the present paper and~\cite{Pokern:2012fk} as special cases. 
Other possible directions are the consideration of nonlinear inverse problems, the use of non-Gaussian priors and/or noise and the extension of the credibility analysis presented in~\cite{1232.62079} to a more general setting.



%% References


%\bibliographystyle{model1b-num-names}
%\bibliography{paperbibnopar.bib}






% \end{thebibliography}

\end{document}